\documentclass[11pt,reqno,a4paper]{amsart}

\title[A degree one Carleson operator along the paraboloid]{A degree one Carleson operator along the paraboloid}
\author[Lars Becker]{Lars Becker}
\address{Mathematical Institute, University of Bonn and Mathematics Department, Princeton University, Princeton, NJ 08544, USA}
\email{lbecker@math.princeton.edu}
\date{\today}

\usepackage[utf8]{inputenc}
\usepackage{hyperref}
\hypersetup{colorlinks=true,
    linkcolor=blue,
    filecolor=magenta,      
    urlcolor=cyan}

\usepackage{mathtools}
\usepackage{amssymb}
\usepackage{amsthm}
\usepackage{enumitem}
\usepackage{xcolor}
\usepackage{verbatim}

\theoremstyle{plain}
\newtheorem{theorem}{Theorem}
\newtheorem{lemma}{Lemma}

\newtheorem{corollary}{Corollary}
\newtheorem{proposition}{Proposition}

\theoremstyle{definition}

\theoremstyle{remark}

\newcommand{\R}{\mathbb{R}}

\DeclareMathOperator{\dens}{\operatorname{dens}}

\DeclareMathOperator{\tp}{\operatorname{top}}

\usepackage[style=trad-alpha]{biblatex}
\bibliography{bib.bib}

\subjclass[2020]{42B20}
\keywords{}

\begin{document}

\begin{abstract}
    We prove $L^p$ bounds, $\frac{d^2 + 4d + 2}{(d+1)^2} < p < 2(d+1)$, for maximal linear modulations of singular integrals along paraboloids with frequencies in certain subspaces of $\R^{d+1}$, for $d \geq 2$. This generalizes Carleson's theorem on convergence of Fourier series, and complements a corresponding result by Pierce and Yung with polynomial modulations without linear terms.
\end{abstract}

\maketitle

\section{Introduction}
This paper advances the program of Pierce and Yung \cite{Pierce2019} of studying maximally modulated  singular Radon transforms along paraboloids.
While their work 
focuses on certain 
polynomial modulations without linear terms and uses $TT^*$ methods, our result is the first instance in this
program for degree one polynomials featuring
symmetries that mandate the use of time-frequency analysis. 

Our main result is as follows.
An $m$-Calderón-Zygmund kernel on $\R^d$ is a function $K: \R^d \setminus \{0\} \to \mathbb{C}$ satisfying the estimates 
\begin{equation}
    \label{eq k size}
    |\partial^{\alpha} K(x)| \leq |x|^{-d- |\alpha|}\,, \quad |\alpha| \leq m\,,
\end{equation}
and the cancellation property
\begin{equation}
    \label{eq k canc}
    \int_{B(0, R) \setminus B(0, r)} K(x) \, \mathrm{d}x = 0\,, \quad 0 < r < R\,.
\end{equation}
Let $V \subset \R^{d+1}$ be a linear subspace.
We consider the maximally modulated singular integral along the paraboloid defined a priori on Schwartz functions $f$ on $\R^{d+1}$ by
\begin{equation}
    \label{eq operator}
    T_Vf(x) = \sup_{N \in V} \sup_{r < R} \left\lvert   \int_{r < |y| < R} f(x' - y, x_{d+1} - |y|^2) e^{iN \cdot (y, |y|^2)} K(y) \, \mathrm{d}y \right\rvert\,,
\end{equation}
where $x = (x', x_{d+1}) \in \R^{d}\times \R = \R^{d+1}$, and $K$ is a Calderón-Zygmund kernel. 

\begin{theorem}
    \label{main Lp}
    Let $d \geq 2$ and let $m > \frac{d}{2}$. Suppose that $V = \{0\}^d \times \R$ or $V$ is a proper subspace of $\R^{d}\times\{0\}$. Then for all $p$ with
    \begin{equation}
    \label{eq p condition}
        \frac{d^2 + 4d + 2}{(d+1)^2} < p < 2(d+1)
    \end{equation}
    there exists $C > 0$ such that for all $m$-Calderón-Zygmund kernels $K$ and all Schwartz functions $f$, we have with $T_V$ as defined in  \eqref{eq operator}
    $$
        \|T_Vf\|_{L^{p}(\R^{d+1})} \leq C\|f\|_{L^{p}(\R^{d+1})}\,.
    $$
\end{theorem}

Note that the singular integral along the paraboloid in \eqref{eq operator} is given by convolution with the tempered distribution $k(z) = \delta(z_{d+1} - |z|^2) K(z)$ on $\R^{d+1}$. In terms of the argument $z$, the modulation argument  $N \cdot (y, |y|^2) = N \cdot z$ in \eqref{eq operator} is a degree one polynomial. This, and its consequences for the method of proof below, is why we call $T_V$ a degree one operator.

\subsection{Motivation}
Our interest in the operator \eqref{eq operator} stems from the following result of Pierce and Yung \cite{Pierce2019}, see also \cite{Anderson2022}. They prove $L^p$ bounds for $p \in (1,\infty)$ for maximally \emph{polynomially} modulated singular integral operators along the paraboloid. More precisely, they consider the operator
\begin{equation}
    \label{eq PY operator}
    f \mapsto \sup_{P} \left| \int_{\R^d} f(x'-y, x_{d+1} - |y|^2) e^{iP(y)} K(y) \, \mathrm{d}y \right|\,,
\end{equation}
where $d \ge 2$ and $P$ ranges over a certain set of polynomials of fixed degree without linear terms, and without a monomial $c|y|^2$.   Note that this excludes exactly the monomials that are present in \eqref{eq operator}. Very recently, Beltran, Guo and Hickman \cite{beltran+2024} gave a  version of the Pierce-Yung theorem with $d = 1$, and $P$ ranging over $\{cy^3 \, : \, c \in \R\}$.

The study of maximally modulated singular integrals such as \eqref{eq operator}, \eqref{eq PY operator} has a long history, starting with Carleson's \cite{Carelson1966} proof of pointwise almost everywhere convergence of Fourier series of $L^2$ functions. His proof relies crucially on an $L^2$ to $L^{2,\infty}$ estimate for the maximally modulated Hilbert transform
\begin{equation}
    \label{eq Carleson operator}
    f \mapsto \sup_{N \in \R} \left| \int f(x - y) e^{iNy} \frac{1}{y} \, \mathrm{d}y \right|\,.
\end{equation}
Carleson's theorem was subsequently extended to $L^p$ for $p \in (1,\infty)$ by Hunt \cite{Hunt1968}, and to singular integrals in higher dimension by Sjölin \cite{Sjoelin1971}.
Other essentially different proofs of Carleson's theorem were later given by Fefferman \cite{Fefferman1973} and by Lacey and Thiele \cite{Lacey+2000}. Endpoint questions were considered in \cite{Antonov1996,LieL1,diplinio2022}.

Pierce and Yung's theorem is motivated by a variation of this theme due to Stein and Wainger \cite{Stein+2001}. They investigated maximally polynomially modulated singular integrals
\begin{equation}
    \label{eq SW operator}
    f \mapsto  \sup_{P} \left| \int f(x-y) e^{iP(y)} K(y) \, \mathrm{d}y \right|\,,
\end{equation}
where $P$ ranges over the set of all polynomials of fixed degree \emph{without linear terms}. They prove $L^p$ bounds for $p \in (1,\infty)$.  
The harder  extension to the operator \eqref{eq SW operator} with $P$ ranging over \emph{all} polynomials of fixed degree was accomplished by Lie \cite{Lie2009, Lie2020}. Zorin-Kranich subsequently gave a version of Lie's proof with very weak regularity assumptions in \cite{ZK2021}.

\subsection{Modulation symmetries}
We want to discuss the relevance of excluding linear terms in the polynomials in \eqref{eq PY operator} and \eqref{eq SW operator}.
The methods used by Stein-Wainger to bound the operator \eqref{eq SW operator}, without linear terms in the modulations, are fundamentally different from the methods employed by Carleson and Sjölin for the same operator with only linear terms. Stein and Wainger use a $TT^*$ argument exploiting almost orthogonality of contributions of different scales, and decay when the polynomial modulation is large \cite[Theorem 1]{Stein+2001}. Sjölin and Lie on the other hand use time frequency analysis.
This difference in methods is dictated by the symmetries of the operator. Carleson's operator \eqref{eq Carleson operator}, and Sjölin's higher dimensional variant, are invariant under the modulations
\begin{equation}
    \label{eq mod sym}
    f \mapsto [x \mapsto e^{iNx}f(x)]\,, \quad N \in \R^d\,.
\end{equation}
On the other hand, a quick computation shows that the operator \eqref{eq SW operator} has no symmetries under the transformations \eqref{eq mod sym}, as long as no linear terms are present in the polynomials $P$. To illustrate how this symmetry affects the proof, the reader is invited to use the modulation symmetry under \eqref{eq mod sym} to show that \cite[Theorem 1]{Stein+2001} would not be true for degree one polynomials. 

Pierce and Yung sidestep the issue of modulations symmetries via their restrictions on the linear and quadratic terms of the polynomials. This allows for a $TT^*$ argument in the same spirit as in \cite{Stein+2001}, but using more sophisticated oscillatory integral estimates. In contrast, our operators $T_V$ are invariant under linear modulations \eqref{eq mod sym} with $N \in V$, so our proof will use time frequency analysis.

\subsection{Previous results}
Pierce and Yung's bound for \eqref{eq PY operator} sparked interest in the corresponding operators with linear modulations. 

Roos \cite{Roos2019} proved a version of Sjölin's theorem for anisotropic Calderón-Zygmund operators with scaling symmetry preserving the paraboloid. Roos's result assumes enough regularity of the Calderón-Zygmund operator, and his bounds blow up if one approximates a singular integral along a paraboloid by such operators. In \cite{Becker2022}, the author improved the control of this blowup. 

Guo, Roos, Pierce and Yung proved in \cite{Guo+2017} a number of interesting related results for Hilbert transforms on planar curves. They weaken the maximal operator by first taking an $L^p$ norm in one of the variables, before taking the supremum in the modulation parameter. Most relevant to us is the following special case of their Theorem 1.2:
\begin{equation}
    \label{eq GPRY operator}
    \int \sup_{N \in \R} \int \left| \int f(x_1 - y, x_2 - y^2) e^{iNy^2} \frac{\mathrm{d}y}{y} \right|^p \, \mathrm{d}x_1 \, \mathrm{d}x_2 \le C \|f\|_p^p\,,
\end{equation}
where $p \in (1, \infty)$.
They also noted that for $p=2$, \eqref{eq GPRY operator} holds with the order of $\mathrm{d}x_1\, \mathrm{d}x_2$ reversed, and also with modulations $e^{iNy}$, by a combination of Plancherel and Lie's polynomial Carleson theorem.

Finally, Ramos \cite{Ramos2021} proved bounds for maximal modulations of the one dimensional Fourier multipliers obtained by restricting the multiplier of the Hilbert transform along the parabola to a line, uniformly in the choice of line.

Our Theorem \ref{main Lp} is the first $L^p$ estimate for maximal linear modulations of a singular integral operator along a submanifold, with - differently from \eqref{eq GPRY operator} - the supremum fully inside the $L^p$ norm.  We cannot directly compare our result to \cite{Guo+2017}, because they consider singular integrals along planar curves, while Theorem \ref{main Lp} assumes that the dimension of the paraboloid is at least two.
However, Theorem \ref{main Lp} implies a stronger version of the hypothetical generalization of \eqref{eq GPRY operator} to paraboloids of dimension $\ge 2$ in the range of $p$ given by \eqref{eq p condition}. We stress that our proof does not apply to the one dimensional parabola, so we do not recover \eqref{eq GPRY operator} itself. However, our Theorem \ref{main Lp} does recover, via projection and limiting arguments,  Sjölin's theorem \cite{Sjoelin1971} in the full range of exponents $p \in (1,\infty)$. In particular, it implies the Carleson-Hunt theorem.

\subsection{Overview of the proof of Theorem \ref{main Lp}}
Our proof is an adaptation of Fefferman's \cite{Fefferman1973} proof of Carleson's theorem to the setting of singular integrals along paraboloids. In some parts we also follow the presentation in \cite{ZK2021}. We establish variants of Fefferman's key time-frequency analysis estimates in this more singular setting. 

The proof starts with a discretization of the operator, in Section \ref{sec discretization}, and a combinatorial decomposition into so called forest operators and antichain operators in Section \ref{sec organization} (see Sections \ref{sec discretization} and \ref{sec outline} for the definitions of these operators). These two steps are straightforward variation of the corresponding steps in Fefferman's proof. However, differently from Fefferman, we test the operator with the indicator function of a set $F$ and use a modified decomposition adapted to $F$. This allows us to directly prove weak type $L^2$ bounds.

After the decomposition of the operator, the argument proceeds by proving bounds for antichain operators in Section \ref{sec antichain}, for so called tree operators in Section \ref{sec tree} and finally for forest operators in Section \ref{sec forest}, see Propositions \ref{prop antichain}, \ref{prop forest} and \ref{prop tree}.
Each of these three steps needs new ingredients for singular integrals supported on submanifolds.

\subsubsection{Antichains: A square function argument}
To control the antichain operators in Section \ref{sec antichain}, we decompose the kernel $k$ with singular support on the paraboloid into a smoothened kernel, which is no longer supported just on a submanifold, and a remainder. The versions of the operator with smoothened kernel are estimated using the argument of Fefferman (see Lemma 1 and Lemma 2 in \cite{Fefferman1973}). The contribution of the remainder is estimated by a square function. This square function is closely related to a classical (unmodulated) square function used in \cite{Stein+1978} to prove boundedness of the maximal average along the paraboloid. In our square function, there is an additional summation over modulations  compared to the one in \cite{Stein+1978}. To bound this larger square function, we rely on the precise decay rate in each direction of the Fourier transform of the kernel supported on the paraboloid, coming from stationary phase. In addition, we use that the Fourier transform of the remainder vanishes near a certain subspace. 

\subsubsection{Trees: Sparse bounds}
In Section \ref{sec tree} we estimate the tree operators. These are pieces of the operator which can be modeled by truncated singular integral operators. 
Fefferman's argument for this step (see Lemma 3 in \cite{Fefferman1973}) relies on the fact that the  convolution of a function with a truncated singular integral kernel is essentially constant at the scale of the lower truncation parameter. This makes the integral of said convolution over a set which is thin at this scale small. In contrast, truncated singular integrals along paraboloids are not essentially constant at the lower truncation scale. Our new ingredient to solve this issue is a Sobolev smoothing estimate for truncated singular integrals along paraboloids. On a technical level, we use certain sparse bounds for singular Radon transforms, see Lemma \ref{lem sparse}, due to Oberlin \cite{Oberlin2019}, see also \cite{lacey2018sparse}. The arguments in Section \ref{sec tree} are formulated for a general class of singular Radon transforms, similar to the setting of \cite{Duoandikoatxea1986}. 

\subsubsection{Forests: Square functions and oscillatory integrals}
In Section \ref{sec forest} we follow Fefferman's argument (Lemma 4 and Lemma 5 in \cite{Fefferman1973}) to prove almost orthogonality estimates between tree operators, and combine the estimates for tree operators to an estimate for forest operators.
We prove that tree operators essentially only act on frequencies close to a central frequency associated to the tree, leading to almost orthogonality for trees with sufficiently separated central frequencies. This requires more work than in \cite{Fefferman1973}, because of the singular support of the kernels on the paraboloid. However, it is within the scope of the square function arguments in \cite{Stein+1978}, \cite{Duoandikoatxea1986}, used there to bound maximal averages and maximally truncated singular integrals along paraboloids. Fefferman's argument requires certain upper bounds to be local, in the sense that they only depend on the values of the functions involved on certain sets associated to the trees. 
To ensure this locality we further use estimates for certain oscillatory integrals along paraboloids, see Section \ref{sub sec aux sep tree}. They replace easier partial integration arguments in \cite{Fefferman1973}.

\medskip
Finally, we deduce $L^p$ bounds in Section \ref{sec Lp}, using interpolation and a localization argument as in \cite{ZK2021}.

\subsection{Limitations and further questions}
The limitations of our methods are still dictated by the symmetries of the operator. The restrictions on $V$ in Theorem \ref{main Lp} ensure that the operator $T_V$ has only linear modulation symmetries. If we had $\R^d \times \{0\} \subset V$, then $T_V$ would be invariant under the transformations
\begin{equation}
    \label{eq nonl mod}
    f \mapsto [(x', x_{d+1}) \mapsto e^{iN(|x'|^2 + x_{d+1})}f(x', x_{d+1})]\,, \quad N \in \R\,.
\end{equation}
Time frequency analysis in its current form seems to be unable to handle such additional symmetries. For example, with our setup Proposition \ref{prop antichain} fails in the presence of the symmetry \eqref{eq nonl mod}, because the definition \eqref{eq densdef} of density used therein is not invariant under \eqref{eq nonl mod}.

Note that for the modulation subspace $V$ to be compatible with the anisotropic dilation symmetry of the paraboloid, it can only be the vertical subspace $\{0\}^d \times \R$, a subspace of $\R^d \times \{0\}$, or a sum of two such spaces. Then the only remaining case (up to rotation around the vertical axis) that could conceivably be within reach of current methods is $V = \R^{d-1} \times \{0\} \times \R$. Our argument does not handle this $V$, the particular point of failure is the square function argument in Subsection \ref{sub sec high antichain}. However, we know of no fundamental obstruction. It is tempting to ask whether the results of this paper can be extended to that case. 

The restriction to paraboloids of dimension two or higher in Theorem \ref{main Lp} is needed in the square function argument in Subsection \eqref{sub sec high antichain}. For the parabola this square function argument fails, because the Fourier transform of a measure on the parabola has too little decay. 

The range \eqref{eq p condition} of $p$ in Theorem \ref{main Lp}, which one might conjecture should really be $(1,\infty)$, is a consequence of restrictions on exponents in the sparse bounds in Lemma \ref{lem sparse}. These restrictions are related to the $L^p$-improving range for averages along paraboloids.
It is an interesting question whether the range of $p$ can be improved by other methods. 
\\

\noindent \textit{Acknowledgement.} The author thanks Christoph Thiele and Rajula Srivastava for various helpful discussions about this paper, and Yifan Zhang for comments on an earlier version. The author was supported by the Collaborative Research Center 1060 funded by the Deutsche
Forschungsgemeinschaft (DFG, German Research Foundation) and the Hausdorff Center for
Mathematics, funded by the DFG under Germany’s Excellence Strategy - GZ 2047/1, ProjectID 390685813.

\section{Reduction to a discretized operator}
\label{sec discretization}

We may assume that $V = \{0\}^d \times \R$ or $V = \R^{d-1} \times \{0\}^2$. 
Define anisotropic dilations $\delta_s(x', x_{d+1}) := (2^sx', 2^{2s}x_{d+1})$.
Define the dyadic cubes of scale $0$ to be
$$
    \mathbf{D}_0 := \{k + [0,1)^{d+1} \, : \, k \in \mathbb{Z}^{d+1}\}\,,
$$
and the dyadic cubes of scale $s$ to be
$
    \mathbf{D}_s := \delta_s(\mathbf{D}_0)\,.
$
The collection of all dyadic cubes is denoted $\mathbf{D} = \cup_{s \in \mathbb{Z}} \mathbf{D}_s$. Given a dyadic cube $I \in \mathbf{D}_s$, we denote by $s(I) = s$ its scale. 

We define dyadic frequency cubes of scale $0$ as 
$$
    \mathbf{\Omega}_0 := \{k + \{0\}^d \times [0,1)\, : \, k \in \{0\}^d \times \mathbb{Z}\}
$$
if $V = \{0\}^d \times \R$ and as
$$
    \mathbf{\Omega}_0 := \{k + [0,1)^{d-1} \times \{0\}^2 \, : \, k \in \mathbb{Z}^{d-1} \times \{0\}^2\}\,,
$$
if $V = \R^{d-1} \times \{0\}^2$. The dyadic frequency cubes of scale $s$ are $\mathbf{\Omega}_s := \delta_{-s}(\mathbf{\Omega}_0)$ and the collection of all dyadic frequency cubes is defined by $\mathbf{\Omega} := \bigcup_{s \in \mathbb{Z}} \mathbf{\Omega}_s$. A tile is a pair $p = (I, \omega)$, where $I \in \mathbf{D}_s$ and $\omega \in \mathbf{\Omega}_s$ for some $s = s(p)$ called the scale of $p$.
The collection of all tiles is denoted by
$$
    \mathbf{P} := \bigcup_{s \in \mathbb{Z}} \{(I, \omega) \, : \, I \in \mathbf{D}_s,\,\omega \in \mathbf{\Omega}_s\}\,.
$$

We decompose the kernel $K$ into pieces localized in dyadic annuli. Fix a smooth function $\eta$ supported in $[1/8, 1/3]$ such that 
$$
    \sum_{s \in \mathbb{Z}} \eta(2^{-s}t) = 1\,,\quad t \in (0, \infty)\,,
$$
and define $K_s(x) = K(x) \eta(2^{-s}|x|)$.
Then \eqref{eq k size} and \eqref{eq k canc} imply that there exists a constant $C = C(\eta)$ with
\begin{equation}
    \label{eq ks size}
    |\partial^\alpha K_s(x)| \leq C 2^{-s(d + |\alpha|)}\,, \quad |\alpha| \leq m\,,\, s\in \mathbb{Z}\,,
\end{equation}
and that 
\begin{equation}
    \label{eq ks canc}
    \int K_s(x) \, \mathrm{d}x = 0\,, \quad s \in \mathbb{Z}\,.
\end{equation}
It suffices to estimate the operator 
\begin{equation}
    \label{eq dis op}
    \sup_{N \in V} \sup_{\underline s < \overline s}\left\lvert  \sum_{s = \underline s}^{\overline s} \int f(x' - y, x_{d+1} - |y|^2) e^{iN \cdot (y, |y|^2)} K_s(y) \, \mathrm{d}y \right\rvert\,,
\end{equation}
since the difference to $T_V$ is controlled by the maximal average along the paraboloid, by \eqref{eq k size}.
Since the sum of integrals in \eqref{eq dis op} is continuous in $N$, we may restrict the supremum in $N$ to a countable dense subset of $V$. By monotone convergence, we may further restrict both suprema  to finite subsets of $V$ and $\mathbb{Z} \times \mathbb{Z}$. Choosing maximizers for each $x$, it suffices to bound the operator
\begin{equation}   
    \label{eq lin}
    \sum_{s = \underline s(x)}^{\overline{s}(x)} \int f(x' - y, x_{d+1} - |y|^2) e^{-iN(x) \cdot (y, |y|^2)} K_s(y) \, \mathrm{d}y
\end{equation}
uniformly over all measurable functions $N: \R^{d+1} \to V, \underline{s}: \R^{d+1} \to \mathbb{Z}$ and $\overline{s}: \R^{d+1} \to \mathbb{Z}$ with finite range. We fix such functions $N, \underline s, \overline s$.

For a tile $p \in \mathbf{P}$, we define 
$$
    E(p) = \{x \in I(p) \, : \, N(x) \in \omega(p), \, \underline{s}(x) \leq s(p) \leq \overline{s}(x)\}
$$
and the operator associated to the tile $p$  
$$
    T_p f(x) = \mathbf{1}_{E(p)} \int f(x' - y, x_{d+1} - |y|^2) e^{-iN(x)\cdot(y,  |y|^2)} K_{s(p)}(y) \, \mathrm{d}y\,.
$$
For a subset $\mathbf{C} \subset \mathbf{P}$ we write $T_\mathbf{C} = \sum_{p \in \mathbf{C}} T_p$.
The operator \eqref{eq lin} is then simply $T_\mathbf{P}$.

Finally, we remove some tiles with unfavourable properties, using an argument due to Fefferman \cite{Fefferman1973}.
A tile is called admissible if $3\omega(p) \subset \hat{\hat{\omega}}(p)$, where $\hat{\omega}$ denotes the unique frequency cube of scale $s(\omega) - 1$ containing $\omega$. Then it suffices to show the estimate $\|T_{\mathbf{P}_\text{ad}}f\|_{L^{p}} \leq C \|f\|_{L^{p}}$, where $\mathbf{P}_\text{ad}$ is the set of admissible tiles, 
as follows from an averaging argument analogous to the one in Section 5 of \cite{Fefferman1973}. From now on by tile we always mean an admissible tile.

\section{Outline of the proof of weak type \texorpdfstring{$L^2$}{L2}-bounds for \texorpdfstring{\eqref{eq lin}}{(8)}}
\label{sec outline}

By duality and the reductions in the previous section, it suffices to show that there exists $C> 0$ such that for each compact set $F \subset \R^{d+1}$, there exists a set $\tilde F \subset F$ with $|\tilde F| \leq \frac{1}{2}|F|$ and
\begin{equation}
    \label{eq goal}
    \| \mathbf{1}_{F \setminus \tilde F} T_{\mathbf{P}_\text{ad}} \|_{2 \to 2} \leq C \,.
\end{equation}
We fix $F$. The operators $\mathbf{1}_F T_p$ do not change upon replacing $E(p)$ by $E(p) \cap F$, so we will assume from now on that $E(p) \subset F$ for all tiles $p$. Note that since $F$ is compact and $\underline s, \overline s$ and $N$ have finite range, only finitely many operators $T_p$ are not zero.

We now make some definitions to state a decomposition for the operator on the left hand side of \eqref{eq goal}. 
Denote by $\dim_h V$ the homogeneous dimension of $V$ with respect to the dilations $\delta_s$, that is $$\dim_h(\R^{d-1} \times \{0\}^2) = d-1\,, \qquad  \qquad \dim_h (\{0\}^d \times \R) = 2\,,$$
and define the convex cylinder associated to a collection $\mathbf{C}$ of tiles by
$$
    C(\mathbf{C}) = \{p \in \mathbf{P}_{\text{ad}} \, : \, \exists p', p'' \in \mathbf{C}, I(p') \subset I(p) \subset I(p'')\}\,.
$$
Then the density of a collection $\mathbf{C}$ of tiles is defined as 
\begin{equation}
    \label{eq densdef}
    \dens(\mathbf{C}) = \sup_{p \in \mathbf{C}} \sup_{\substack{\lambda \geq 1\\\text{odd}}} \sup_{\substack{p' \in C(\mathbf{C})\,:\, I(p) \subset I(p')\\ \lambda \omega(p') \subset \lambda \omega(p)}}  \lambda^{-\dim_h V} \frac{| E(\lambda, p')|}{|I(p')|}\,,
\end{equation}
where we use the notation $\lambda \omega = c(\omega) + \delta_{\log_2 \lambda}(\omega - c(\omega))$ where $c(\omega)$ is the center of $\omega$ and
$$
    E(\lambda, p') = \{x \in I(p)\cap F \, : \, N(x) \in \lambda \omega(p)\}\,.
$$

We define a partial order on the set of tiles by 
$$(I, \omega) \leq (I', \omega') \iff I \subset I' \ \text{and} \ \omega' \subset \omega\,.$$
An antichain is a set of tiles that are pairwise not comparable with respect to this order. A set of tiles $\mathbf{C}$ is called convex, if $p, p' \in \mathbf{C}$ and $p \leq p'' \leq p'$ implies $p'' \in \mathbf{C}$. A tree is a convex collection $\mathbf{T}$ of tiles together with an upper bound $\tp(\mathbf{T})$, i.e. a tile $m$ such that for all $p \in \mathbf{T}$ we have $p \leq m$. We denote $\omega(\mathbf{T}) = \omega(\tp(\mathbf{T}))$ and  $I(\mathbf{T}) = I(\tp(\mathbf{T}))$. A tree is called normal if $3I(p) \subset I(\mathbf{T})$ for all $p \in \mathbf{T}$. 

A pair of trees $\mathbf{T}_1, \mathbf{T}_2$ is called $\Delta$-separated, if for $\{i,j\} = \{1,2\}$ and each tile $p_i \in \mathbf{T}_i$ with $I(p_i) \subset I(\mathbf{T}_j)$ we have $\Delta \omega(p_i) \cap \omega(\mathbf{T}_j) = \emptyset$. 
An $n$-forest is a collection of pairwise $2^{10dn}$-separated, normal trees of density at most $2^{-n}$ that satisfy the overlap estimate 
\begin{equation}
    \label{eq forest overlap}
    \sum_{\mathbf{T} \in \mathbf{F}} \mathbf{1}_{I(\mathbf{T})} \leq 2^n \log(n+2)\,.
\end{equation}

\begin{proposition}
    \label{prop organization}
    There exists a constant $C = C(d)$ and an exceptional set $\tilde F$ with $|\tilde F| \leq  |F|/2$, such that the set 
    $$
        \mathbf{P}_{F\setminus \tilde F} := \{p \in \mathbf{P}_\mathrm{ad} \, : \, I(p) \cap (F \setminus \tilde F) \neq \emptyset\}
    $$ 
    can be decomposed as a disjoint union 
    $$
        \mathbf{P}_{F\setminus \tilde F} = \bigcup_{n \geq 0} \left(\bigcup_{l=1}^{C(n+1)^2} \mathbf{F}_{n,l} \cup \bigcup_{l=1}^{C(n+1)^3} \mathbf{A}_{n,l}\right)\,,
    $$
    where each $\mathbf{F}_{n,l}$ is an $n$-forest and each $\mathbf{A}_{n,l}$ is an antichain of density at most $2^{-n}$.
\end{proposition}

Estimate  \eqref{eq goal}, and therefore weak type $L^2$ boundedness of the operator defined in \eqref{eq operator} then follows from the following estimates for antichain and forest operators. 

\begin{proposition}
    \label{prop antichain}
    There exists $\varepsilon = \varepsilon(d)$ and $C > 0$ such that the following holds.
    Let $\mathbf{A}$ be an antichain of density $\delta$. Then
    $$
        \|T_{\mathbf{A}}\|_{2 \to 2} \leq C \delta^{\varepsilon}\,. 
    $$
\end{proposition}

\begin{proposition}
    \label{prop forest}
    For each $\varepsilon < \frac{1}{2} - \frac{1}{2(d+1)}$ there exist $C > 0$ such that the following holds. Let $\mathbf{F}$ be an $n$-forest. Then 
    $$
        \|T_\mathbf{F}\|_{2 \to 2} \leq C2^{-\varepsilon n}\,.
    $$
\end{proposition}

We will independently prove Proposition \ref{prop organization} in Section \ref{sec organization} and Proposition \ref{prop antichain} in Section \ref{sec antichain}. Proposition \ref{prop forest} is proven using an estimate for trees and an almost orthogonality argument. The almost orthogonality argument and the deduction of Proposition \ref{prop forest} are carried out in Section \ref{sec forest}. The estimate for single trees is proven in Section \ref{sec tree}, and is a mild generalization of the following. 

\begin{proposition}
    \label{prop tree}
    For each  $\varepsilon < \frac{1}{2} - \frac{1}{2(d+1)}$ there exist $C > 0$ such that the following holds. Let $\mathbf{T}$ be a tree of density $\delta$. Then 
    $$
        \|T_\mathbf{T}\|_{2 \to 2} \leq C\delta^{\varepsilon}\,.
    $$
\end{proposition}

\section{Tile organization: Proof of Proposition \ref{prop organization}}
\label{sec organization}

Let $k \geq 0$ and let $\mathbf{D}_k(F)$ be the set of maximal dyadic cubes $Q$ with $|Q \cap F|/|Q| \geq 2^{-k-1}$. Let $\tilde{\mathbf{P}}_{\leq k}$ be the set of  tiles $p \in \mathbf{P}_{\text{ad}}$ such that $I(p)$ is contained in some $Q \in \mathbf{D}_{k}(F)$ and let $\tilde{\mathbf{P}}_k := \tilde{\mathbf{P}}_{\leq k} \setminus \tilde{\mathbf{P}}_{\leq k-1}$.
Then we have $\mathbf{1}_F T_{\mathbf{P}_\text{ad}} = \sum_{k \geq 0} \mathbf{1}_F T_{\tilde{\mathbf{P}}_k}$, and each $p \in \tilde{\mathbf{P}}_k$ satisfies that $|I(p) \cap F|/|I(p)| < 2^{-k}$. We define 
$$
    \overline{E}(p) := \{ x \in I(p) \cap F \, : \, N(x) \in \omega(p)\}\,,
$$
and define $\tilde{\mathbf{M}}_{n,k}$ to be the set of maximal tiles $p$ in $\tilde{\mathbf{P}}_k$ such that $|\overline{E}(p)|/|I(p)| \geq 2^{-n-1}$. 
\begin{lemma}
    \label{lem ex}
    The exceptional set
    $$
        \tilde F_1 := F \cap \bigcup_{k \geq 0} \bigcup_{Q \in \mathbf{D}_k(F)} \bigcup_{n \geq k} \{x \in Q \, : \, \sum_{\substack{p \in \tilde{\mathbf{M}}_{n,k}\\ I(p) \subset Q}} \mathbf{1}_{I(p)} \geq 1000 \cdot 2^n\log(n+2)\}
    $$
    satisfies $|\tilde F_1| \leq |F|/4$.
\end{lemma}

\begin{proof}
We have for each $J \in \mathbf{D}$
$$
    \sum_{\substack{p \in \tilde{\mathbf{M}}_{n,k}\\ I(p) \subset J}} |I(p)| \leq 2^{n+1} \sum_{\substack{p \in \tilde{\mathbf{M}}_{n,k}\\ I(p) \subset J}} |\overline{E}(p)| \leq 2^{n+1} |J|\,,
$$
since tiles in $\tilde{\mathbf{M}}_{n,k}$ are pairwise not comparable and hence the sets $\overline{E}(p)$, $p \in \tilde{\mathbf{M}}_{n,k}$ are pairwise disjoint. 
We estimate for each $Q \in D_k(F)$ and each $n \geq k$ using the John-Nirenberg inequality:
\begin{multline}
\label{eq JN appl}
    |\{x \in Q \, : \, \sum_{\substack{p \in \tilde{\mathbf{M}}_{n,k}\\ I(p) \subset Q}} \mathbf{1}_{I(p)} \geq 1000 \cdot 2^n\log(n+2)\}|\\
    \leq c_1  \exp(-c_2 \frac{1000 \cdot 2^n \log(n+2)}{2^{n+1}})|Q|\leq (n + 2)^{-100}|Q|\,.
\end{multline}
For the last inequality we have used that in this version of the John-Nirenberg inequality, one can choose $c_1 = e^2$ and $c_2 = (2e)^{-1}$.

Note that the set on the left hand side of \eqref{eq JN appl} is a disjoint union of cubes $I(p)$ each of which satisfies $|I(p) \cap F|/|I(p)| \leq 2^{-k}$, hence 
\begin{multline*}
    |F \cap \{x \in Q \, : \, \sum_{\substack{p \in \tilde{\mathbf{M}}_{n,k}\\ I(p) \subset Q}} \mathbf{1}_{I(p)} \geq 1000 \cdot 2^n\log(n+2)\}| \\
    \leq 2^{-k}(n+2)^{-100}|Q|\leq 2(n+2)^{-100} |Q \cap F|\,.
\end{multline*}
Summing up, we obtain
\begin{multline*}
    \sum_{k \geq 0} \sum_{Q \in \mathbf{D}_k(F)} \sum_{n \geq k} |F \cap \{x \in Q \, : \, \sum_{\substack{p \in \tilde{\mathbf{M}}_{n,k}\\ I(p) \subset Q}} \mathbf{1}_{I(p)} \geq 1000 \cdot 2^n\log(n+2)\}| \\
    \leq 2\sum_{k \geq 0} \sum_{n \geq k} (n+2)^{-100} |F| \leq \frac{1}{4} |F|\,.
\end{multline*}
This completes the proof.
\end{proof}

After removing the exceptional set $\tilde F_1$, only the tiles in 
$$
    \mathbf{P}_k := \{p \in \tilde{\mathbf{P}}_k \, : \, I(p) \not\subset \tilde F_1\}
$$
contribute, i.e. we have $\mathbf{1}_{\R^{d+1} \setminus \tilde F_1} T_{\tilde{\mathbf{P}}_k} = \mathbf{1}_{\R^{d+1} \setminus \tilde F_1} T_{ \mathbf{P}_k}$. Thus it suffices to decompose the sets $\mathbf{P}_k$ into forests and antichains.

We define the $k$-density of a tile $p \in \mathbf{P}_k$ to be 
$$
    \dens_k(p) := \sup_{\substack{\lambda \geq 1\\\text{odd}}} \sup_{\substack{p' \in \mathbf{P}_k: I(p) \subset I(p')\\ \lambda \omega(p') \subset \lambda \omega(p)}}  \lambda^{-\dim_h V} \frac{| E(\lambda, p')|}{|I(p')|}\,.
$$
Then we split each $\mathbf{P}_k$ into sets $\mathbf{H}_{n,k} := \{p\in\mathbf{P}_k \, : \, 2^{-n-1} < \dens_k(p) \leq 2^{-n}\}$, and decompose each of them separately into forests and antichains. Note that $\dens(\mathbf{C}) \leq \sup_{p \in \mathbf{C}} \dens_k(p)$ for each $\mathbf{C} \subset \mathbf{P}_k$, so all forests and antichains obtained in the decomposition of $\mathbf{H}_{n,k}$ have density at most $2^{-n}$.

\begin{lemma}
    \label{lem nk dec}
    For each $k \geq 0$ and $n \geq k$, there exists an exceptional set $\tilde F_{n,k}$ such that $|\tilde F_{n,k}| \leq 2^{-k-n-2}|F|$, and such that the set of tiles $p \in \mathbf{H}_{n,k}$ with $I(p) \not\subset \tilde{F}_{n,k}$ can be decomposed as a disjoint union of $O(n+1)$ many $n$-forests and $O_d((n+1)^2)$ many antichains.
\end{lemma}

\begin{proof}
    We first note that $\mathbf{H}_{n,k}$ is convex: Since $\dens$ is a decreasing function on tiles with respect to $\leq$, $\mathbf{H}_{n,k}$ is the difference of two down-sets, and as such convex.
    
    Next, we prune the top $n+2$ layers off $\mathbf{H}_{n,k}$: Let $\mathbf{H}_{n,k}^+$ be the set of tiles $p \in \mathbf{H}_{n,k}$ for which there exists no chain $p < p_1 < \dotsb < p_{n+2}$ with all $p_i \in \mathbf{H}_{n,k}$, where we use $p<p'$ to say that $p\leq p'$ and $p \neq p'$. Clearly, $\mathbf{H}_{n,k}^+$ is the union of at most $n+2$ antichains, so it suffices to decompose $\mathbf{H}_{n,k}^0 := \mathbf{H}_{n,k} \setminus \mathbf{H}_{n,k}^+$, and this set is still convex.

    For $n\geq k \geq 0$, let $\mathbf{M}_{n,k}$ to be the set of maximal tiles $p$ in $\mathbf{P}_k$ that satisfy $|\overline{E}(p)|/|I(p)| \geq 2^{-n-1}$. By Lemma \ref{lem ex} and the definition of $\mathbf{P}_k$, we then have the overlap estimate
    \begin{equation}
        \label{eq overlap}
        \sum_{p \in \mathbf{M}_{n,k}} \mathbf{1}_{I(p)} \leq 1000 \cdot 2^n \log(n+2)\,.
    \end{equation}

    We claim that for each $p \in \mathbf{H}_{n,k}^0$, there exists $m \in \mathbf{M}_{n,k}$ with $p \leq m$. Since $p \in \mathbf{H}_{n,k}^0$, there exists a chain $p < p_1 < \dotsb < p_{n+2}$ with $p_{n+2} \in \mathbf{H}_{n,k}$. Since all tiles are admissible, we have $3^{\lfloor (n+2)/2\rfloor} \omega(p_{n+2}) \subset \omega(p)$. 
    Since $p_{n+2} \in \mathbf{H}_{n,k}$, there exists an odd $\lambda \geq 1$ and a tile $p' \in \mathbf{P}_k$ with $I(p_{n+2}) \subset I(p')$ and $\lambda \omega(p') \subset \lambda \omega(p_{n+2})$ such that
    \begin{equation}
        \label{eq use density}
        \frac{|E(\lambda, p')|}{|I(p')|} \geq \lambda^{\dim_h V} 2^{-n-1}\,.
    \end{equation}
    Since $\lambda$ is odd, the set $\lambda \omega(p')$ is the disjoint union of $\lambda^{\dim_h V}$ cubes $\omega(p'')$ of tiles $p''$ with $I(p'') = I(p')$, thus there exists one such $p'' \in \mathbf{P}_k$ with $|\overline{E}(p'')|/|I(p'')| \geq 2^{-n-1}$.
    By definition of $\mathbf{M}_{n,k}$, there is a tile $m \in \mathbf{M}_{n,k}$ with  $p'' \leq m$.
    Equation \eqref{eq use density} implies that $\lambda \leq 2^{(n+1)/\dim_h V}$, so that 
    $$
        \omega(m) \subset \omega(p'') \subset 2^{(n+1)/\dim_h V} \omega(p_{n+2}) \subset \omega(p)\,.
    $$
    Combining all of this we obtain that $p \leq m$, so the claim holds.

    For a tile $p \in \mathbf{H}_{n,k}^0$, let $\mathbf{B}(p)$ be the set of tiles $m \in \mathbf{M}_{n,k}$ with $p \leq m$. Decompose $\mathbf{H}_{n,k}^0$ into $2n + 10$  collections 
    $$
        \mathbf{C}_j := \{p \in \mathbf{H}_{n,k} \, : \, 2^{j-1} \leq |\mathbf{B}(p)| < 2^{j}\}\,, \qquad j = 1, \dotsc, 2n + 10\,.
    $$
    The collections $\mathbf{C}_j$ are convex, since $p \leq p' \leq p''$ implies $B(p'') \subset B(p') \subset B(p)$, and by the overlap estimate \eqref{eq overlap} and the claim, their union is $\mathbf{H}_{n,k}^0$.

    Let $\mathbf{U}_j$ be the set of maximal tiles in $\mathbf{C}_j$, clearly $\mathbf{U}_j$ also has overlap bounded by $100 \cdot 2^n \log(n+2)$. For each $m \in \mathbf{U}_j$ let 
    $$
        \tilde{\mathbf{T}}(m) := \{p \in \mathbf{C}_j \, : \, p \leq m\}\,.
    $$
    Then the sets $\tilde{\mathbf{T}}(m)$ are disjoint: If $p \in \tilde{\mathbf{T}}(m) \cap \tilde{\mathbf{T}}(m')$ for $m \neq m'$, then $\mathbf{B}(m) \cup \mathbf{B}(m') \subset \mathbf{B}(p)$. But the sets $\mathbf{B}(m)$ and $\mathbf{B}(m')$ are disjoint: Else there would be $m''$ with $p \leq m, m' \leq m''$, which implies that $m, m'$ are comparable. But they are both maximal in $\mathbf{C}_j$, so they cannot be comparable.
    Hence $2^{j+1} > |\mathbf{B}(p)| \geq |\mathbf{B}(m)| + |\mathbf{B}(m')| \geq 2^j + 2^j$, a contradiction. In particular, tiles $p \in \tilde{\mathbf{T}}(m)$, $p' \in \tilde{\mathbf{T}}(m')$ for $m \neq m'$ are not comparable.

    The sets $\tilde{\mathbf{T}}(m)$ are of course also convex, so they are trees with top $m$. To obtain the separation property, we prune the bottom $20dn$ layers of the trees: We define $\mathbf{T}^-(m)$ to be the set of tiles $p \in \tilde{\mathbf{T}}(m)$ for which there exists no chain $p_{20dn} < \dotsb < p_1 < p$. Clearly, $\mathbf{T}^-(m)$ is the union of at most $20dn$ antichains. As tiles in different $\tilde{\mathbf{T}}(m)$ are never comparable, $\cup_{m \in \mathbf{U}_j} \mathbf{T}^-(m)$ is still a union of at most $20dn$ antichains. Let $\tilde{\mathbf{T}}^0(m) := \tilde{\mathbf{T}}(m) \setminus \mathbf{T}^-(m)$, this is still a convex tree with top $m$. If $p \in \tilde{\mathbf{T}}^0(m)$, then there exists a chain $p_{20dn} < \dotsb < p_1 < p$ in $\tilde{\mathbf{T}}(m)$. If $m' \neq m$ is such that $I(p) \subset I(m') = I(\tilde{\mathbf{T}}(m'))$, then by the last paragraph we must have $\omega(p_{20dn}) \cap \omega(m') = \emptyset$. Since all tiles are admissible, it follows that $3^{10dn} \omega(p) \cap \omega(m') = \emptyset$. Hence the trees $\tilde{\mathbf{T}}^0(m)$ are $3^{10dn}$-separated.

    Finally, we make the trees normal. For this purpose, let $r = 100 + d + 6n$. We first prune the top $r$ layers $\tilde{\mathbf{T}}^{0+}(m)$ off each $\tilde{\mathbf{T}}^0(m)$ similarly as in the last paragraph. This produces another $r = O_d(n+1)$ antichains.  Then we define the exceptional set 
    $$
        \tilde F_{n,k} =\bigcup_j \bigcup_{m \in \mathbf{U}_j} (I(m) \setminus (1 - 2^{-r})I(m))\,,
    $$
    where $aQ = c(Q) + \delta_{\log_2 a}(Q - c(Q))$ is the anisotropic dilate of $Q$ by a factor $a$ about its center $c(Q)$.
    We have, using that $1 \leq j \leq 2n + 10$, the Bernoulli inequality and \eqref{eq overlap} 
    \begin{align*}
        |\tilde F_{n,k}| &\leq (2n + 10)\cdot \sum_{m \in \mathbf{U}_j} (d+2) 2^{-r} |I(m)|\\
        &\leq (2n + 10)\cdot 1000\cdot 2^n \log(n+2) \cdot (d+2) \cdot 2^{-r} \cdot  \sum_{Q \in \mathbf{D}_k(F)} |Q|\\
        &\leq 2^{-3n-3} \sum_{Q \in \mathbf{D}_k(F)} |Q|\,.
    \end{align*}
    Using the definition of $\mathbf{D}_k(F)$ and that $n \geq k$ we estimate this by
    $$
        \leq 2^{-3n-2} 2^k |F| \leq 2^{-n-k-2} |F|\,.
    $$
    We finally define  
    $$
        \mathbf{T}(m) := \{p \in \tilde{\mathbf{T}}^0(m) \setminus \tilde{\mathbf{T}}^{0+}(m) \, :\, I(p) \not\subset \tilde F_{n,k}\}\,.
    $$
    Then $\mathbf{T}(m), m \in \mathbf{U}_j$ is still a collection of $2^{10dn}$-separated trees, and they are now normal: If $p \in \mathbf{T}(m)$ then $s(p) \leq s(m) - r$ and $I(p) \subset (1 - 2^{-r})I(m)$, and therefore $3I(p) \subset I(m)$. By the overlap estimate \eqref{eq overlap}, the collection $\{\mathbf{T}(m) \, : \, m \in \mathbf{U}_j\}$ is the union of at most $1000$ $n$-forests. Thus $\mathbf{C}_j$ can be decomposed into $1000$ $n$-forests and $O_d(n+1)$ antichains, which completes the proof.
\end{proof}

Recall that all antichains in Lemma \ref{lem nk dec} have density at most $2^{-n}$, since they are contained in $\mathbf{H}_{n,k}$. Taking into account all $k \geq 0$,  Lemma \ref{lem nk dec} then yields a total of $O_d((n+1)^3)$ antichains of density at most $2^{-n}$, a total of $O((n+1)^2)$ many $n$-forests, and an exceptional set $\tilde F_{2} = \bigcup_{n, k} \tilde F_{n,k}$ with $|\tilde F_2| \leq |F|/4$. Combining this with the estimate for the measure of $\tilde F_1$ from Lemma \ref{lem ex}, we obtain Proposition \ref{prop organization} with $\tilde F := \tilde F_1 \cup \tilde F_2$. 

\section{Antichains: Proof of Proposition \ref{prop antichain}}
\label{sec antichain}

\subsection{Decomposition of the kernel}
\label{sub sec dec antichain}
We fix an antichain $\mathbf{A}$ of density $\delta$.
We will decompose $T_\mathbf{A}$, based on a decomposition of the kernel of the singular Radon transform.
Let $\mu_s$ be the measure defined by 
$$
    \int f \, \mathrm{d}\mu_s = \int_{\R^d} f(y,|y|^2) K_s(y) \, \mathrm{d}y\,.
$$
Then the operator $T_p$ can be expressed as 
\begin{equation}
    \label{eq Tp mu}
    T_p f(x) = \mathbf{1}_{E(p)} (x) \int f(x-y) e^{iN(x) \cdot y} \mathrm{d}\mu_{s(p)}(y)\,.
\end{equation}
We now decompose the measures $\mu_s$ into a smooth part and a high-frequency part, choosing different decompositions depending on $V$. For this, we let $\varepsilon_0 > 0$ be a small enough positive number such that the exponent of $\delta$ in Lemma \ref{lem l antichain} below is positive. 

In the case $V = \R^{d-1} \times \{0\}^2$, we let $\varphi^1$ be a smooth bump function supported on $B(0,100^{-1}) \subset \R$ with integral $1$, and let
\begin{equation}
    \label{eq smoothen def 1}
    \varphi_{s,\varepsilon_0}^1(x) = \delta^{-2\varepsilon_0}2^{-2s} \varphi^1(2^{-2s} \delta^{-2\varepsilon_0} x_{d+1}) \delta(x')\,.
\end{equation}
We define the low frequency part of $\mu_s$ by 
\begin{equation}
    \mu_s^{l}(x) := \mu_s * \varphi_{s,\varepsilon_0}^1(x) = \delta^{-2\varepsilon_0} 2^{-2s} K_s(x') \varphi^1(2^{-2s}\delta^{-2\varepsilon_0}(x_{d+1} - |x'|^2)) \label{eq musl formula 1}
\end{equation}
and the high frequency part by $\mu_s^h = \mu_s - \mu_s^l$.
Note that $\varphi_{s,\varepsilon_0}^1$ is a measure supported on the line $x' = 0$, so $\mu_s^l$ is essentially frequency localized near the hyperplane $\xi_{d+1} = 0$.

In the case $V = \{0\}^d \times \R$, we let $\varphi^{d}$ be a smooth bump function supported on $B(0,100^{-1}) \subset \R^{d}$ with integral $1$, and let
\begin{equation}
    \label{eq smoothen def 2}
    \varphi_{s,\varepsilon_0}^d(x) = \delta^{-d\varepsilon_0}2^{-ds} \varphi^d(2^{-s} \delta^{-\varepsilon_0} x')\delta(x_{d+1})\,.
\end{equation}
We define the low frequency part of $\mu_s$ by 
\begin{align}
    \mu_s^{l}(x) &:= \mu_s * \varphi_{s,\varepsilon_0}^d(x)\nonumber\\
    &= \delta^{-d\varepsilon_0} 2^{-ds} \int \delta(|y'|^2 - x_{d+1}) K_s(y') \varphi^d(2^{-s}\delta^{-\varepsilon_0}(x' - y')) \, \mathrm{d}y'  \label{eq musl formula 2}
\end{align}
and the high frequency part by $\mu_s^h = \mu_s - \mu_s^l$. In this case, $\varphi_s^d$ is a measure supported on the hyperplane $x_{d+1} = 0$, and $\mu_s^l$ is essentially frequency localized near the line $\xi' = 0$.

In both cases the low frequency part $\mu_s^l$ is a function, and one can check using \eqref{eq musl formula 1} or \eqref{eq musl formula 2} and \eqref{eq ks size} that
\begin{equation}
    \label{eq musl 1}
    |\mu_s^l(x)| \lesssim \delta^{-d\varepsilon_0}|I(p)|^{-1}
\end{equation}
and
\begin{equation}
    \label{eq musl 2}
    |\partial_i \mu_{s(p)}^l(x)| \lesssim 
    \begin{cases}
        \delta^{-(d+2) \varepsilon_0} 2^{-s} |I(p)|^{-1} & \text{if $i = 1, \dotsc, d$,}\\
        \delta^{-(d+2) \varepsilon_0} 2^{-2s} |I(p)|^{-1} & \text{if $i = d+1$.}\\
    \end{cases}
\end{equation}

We define operators $T_p^l, T_p^h$ by replacing $\mu_{s(p)}$ with $\mu_{s(p)}^l, \mu_{s(p)}^h$ in \eqref{eq Tp mu}, and we define $T_\mathbf{A}^l, T_\mathbf{A}^h$ as the sum of $T_p^l, T_p^h$ over all $p \in \mathbf{A}$. Then we have $T_\mathbf{A} = T_\mathbf{A}^l + T_\mathbf{A}^h$. We prove estimates for $T_\mathbf{A}^l$ and $T_\mathbf{A}^h$ separately in the next two subsections.

\subsection{The smooth part}
\label{sub sec low antichain}

Here we estimate the smooth part $T_\mathbf{A}^l$. The argument is the same for $V = \R^{d-1} \times \{0\}^2$ and $V=\{0\}^d \times \R$. While $\mu_s^l$ is defined differently in these two cases, we will only use the properties \eqref{eq musl 1} and \eqref{eq musl 2} of $\mu_s^l$, which hold in both cases.

We define the separation $\Delta(p, p')$ of a pair of tiles $p, p'$ to be $1$ if $\omega(p) \cap \omega(p') \neq \emptyset$, and else we define it as 
$$
    \Delta(p, p') := \sup\{\lambda > 0 \, : \, \lambda \omega(p) \cap \omega(p') = \emptyset \ \text{and} \ \omega(p) \cap \lambda \omega(p') = \emptyset\}\,.
$$
Then we have the following almost orthogonality estimate.
\begin{lemma}
    \label{lem tile TT*}
    There exists a constant $C>0$ such that for all tiles $p, p'$ with $s(p) \geq s(p')$, we have 
    \begin{equation}
        \label{eq tile TT*}
        \lvert\langle T_p^{l*} g, T_{p'}^{l*} g\rangle\rvert \leq C \delta^{-(2d + 2)\varepsilon_0}\frac{\Delta(p, p')^{-1}}{|I(p)|} \int_{E(p)} |g| \int_{E(p')} |g|\,.
    \end{equation}
\end{lemma}

\begin{proof}
    We abbreviate $\Delta = \Delta(p,p')$.
    We have, by Fubini and the definition of $T_p^l$
    \begin{multline}
        \label{eq first TT*}
        \lvert\langle T_p^{l*} g, T_{p'}^{l*} g\rangle\rvert \leq  \int\int |\mathbf{1}_{E(p)}g(x_1)||\mathbf{1}_{E(p')} g(x_2)| \\\left\lvert \int e^{i(N(x_1) - N(x_2))\cdot y} \mu_{s(p)}^l(x_1 - y) \overline{\mu_{s(p')}^l(x_2 - y)} \, \mathrm{d}y \right| \, \mathrm{d}x_1 \, \mathrm{d}x_2\,.
    \end{multline}
    The inner integral in \eqref{eq first TT*} is bounded by $C \delta^{-2d\varepsilon_0} |I(p)|^{-1}$, by \eqref{eq musl 1}. This implies \eqref{eq tile TT*} if $\Delta \leq 3$, so we will assume from now on that $\Delta > 3$.
    Fix $x_1 \in E(p), x_2 \in E(p')$ and let 
    $$
        \Phi(y) = \mu_{s(p)}^l(x_1 - y) \overline{\mu_{s(p')}^l(x_2 - y)}\,.
    $$
    By \eqref{eq musl 1} and \eqref{eq musl 2}, we have the bounds
    $$
        |\partial_i \Phi(y)| \lesssim \delta^{-(2d+2)\varepsilon_0} \frac{1}{|I(p)||I(p')|} 2^{-s(p')}\,, \quad i = 1, \dotsc, d
    $$
    and 
    $$
        |\partial_i \Phi(y)| \lesssim \delta^{-(2d+2)\varepsilon_0} \frac{1}{|I(p)||I(p')|} 2^{-2s(p')}\,, \quad i = d+1\,.
    $$
    Since $\Delta = \Delta(p, p') > 3$, we have by definition
    $$
        \max_{i=1, \dotsc, d} 2^{s(p)} |N_i(x_1) - N_i(x_2)| \geq  (\Delta -1)/2 \geq \Delta/3
    $$ 
    or 
    $$
        2^{2s(p)} |N_{d+1}(x_1) - N_{d+1}(x_2)| \geq (\Delta^2 - 1)/2 \geq \Delta/3\,.
    $$
    Integrating by parts in the corresponding direction, we find that
    \begin{multline*}
        \left\lvert \int e^{i(N(x_1) - N(x_2))\cdot y} \mu_{s(p)}^l(x_1 - y) \overline{\mu_{s(p')}^l(x_2 - y)} \, \mathrm{d}y \right|\\
        \lesssim \Delta^{-1} \delta^{-(2d + 2) \varepsilon_0} |I(p)|^{-1}\,,
    \end{multline*}
    which together with \eqref{eq first TT*} gives the claimed estimate \eqref{eq tile TT*}.
\end{proof}

\begin{lemma}
    \label{lem ant dens}
    Let $1 \leq q \leq \infty$.
    There exists $C > 0$ such that for every antichain $\mathbf{A}$ of density $\delta$ and every tile $p$
    $$
        \Bigg\|\sum_{p' \in \mathbf{A}, s(p') \leq s(p)} \Delta(p, p')^{-1} \mathbf{1}_{E(p')}\Bigg\|_{q} \leq C \delta^{\frac{\kappa}{q}} \Bigg|\bigcup_{p' \in \mathbf{A}, s(p') \leq s(p)} I(p')\Bigg|^{1/q}\,, 
    $$
    where $\kappa = (1 + \dim_h V)^{-1}$.
\end{lemma}

\begin{proof}
    The estimate holds for $q = \infty$ with $C = 1$, since $\Delta(p, p') \geq 1$ and the sets $E(p')$, $p' \in \mathbf{A}$ are disjoint.
    By Hölder's inequality it therefore only remains to show the estimate for $q = 1$: 
    $$
        \sum_{p' \in \mathbf{A}, s(p') \leq s(p)} \Delta(p, p')^{-1} |E(p')| \lesssim \delta^{\kappa} \Bigg|\bigcup_{p' \in \mathbf{A}, s(p') \leq s(p)} I(p')\Bigg|\,.
    $$
    The contribution of tiles with $\Delta(p, p') \geq \delta^{-\kappa}$ is clearly bounded by the right hand side, since $E(p') \subset I(p')$ and the sets $E(p')$ are disjoint. So it remains only to estimate the contribution of the tiles
    $$
        \mathbf{A}' := \{p' \in \mathbf{A} \, : \, s(p') \leq s(p) \,, \ \Delta(p, p') < \delta^{-\kappa}\}\,.
    $$
    Let $\mathbf{L}$ be the set of maximal dyadic cubes $L$ for which there exists $p' \in \mathbf{A}'$ with $L \subsetneq I(p')$, and there exists no $p' \in \mathbf{A}'$ with $I(p') \subset L$. The set $\mathbf{L}$ is a partition of $\cup_{p' \in \mathbf{A}'} I(p')$, so it will be enough to show for each $L \in \mathbf{L}$ the estimate
    \begin{equation}
        \label{eq L small}
        \bigg|L \cap \bigcup_{p' \in \mathbf{A}'} E(p')\bigg| \leq \delta^{1 - \kappa \dim_h V } |L|\,.
    \end{equation}
    Fix $L \in \mathbf{L}$. There exists a tile $p' \in \mathbf{A}'$ with $I(p') \subset \hat L$. If $I(p') = \hat L$, define $p_L = p'$, and else let $p_L$ be the unique tile with $I(p_L) = \hat L$ and $\omega(p) \subset \omega(p_L)$.  If $\lambda$ is the smallest odd number such that $\lambda \geq 5\delta^{-\kappa}$, then we have in both cases that $ \lambda \omega(p') \supset \lambda \omega(p_L)$. For each $p'' \in \mathbf{A}'$ with $L\cap I(p'') \neq \emptyset$, it holds that $L \subsetneq I(p'')$ and $\omega(p'') \subset \lambda \omega(p_L)$. It follows that for all such $p''$, we have that $L \cap E(p'') \subset E(\lambda, p_L)$. Thus
    \begin{multline*}
        |L \cap \bigcup_{p'' \in \mathbf{A}'} E(p'')| \leq |E(5 \delta^{-\kappa} p_L)| \\\leq \lambda^{-\dim_h V} \dens(\mathbf{A}) |L| \lesssim\delta^{1 - \kappa \dim_h V }|L|\,,
    \end{multline*}
    giving \eqref{eq L small} and hence the lemma. 
\end{proof}

Using the last two lemmas, we can prove our main estimate for $T_\mathbf{A}^l$.

\begin{lemma}
    \label{lem l antichain}
    For all $1 \leq q < 2$, there exists a constant $C> 0$ such that 
    $$\|T_\mathbf{A}^l\|_{2\to2} \leq C \delta^{\frac{\kappa}{2q'} - (d + 1) \varepsilon_0}\,,$$
    where $\kappa = 1/(1+ \dim_h V)$.
\end{lemma}

\begin{proof}
    Define for each $p \in \mathbf{A}$ the set
    $$
        \mathbf{A}(p) = \{p' \in \mathbf{A} \, : \, s(p') \leq s(p), \, 3I(p) \cap 3I(p') \neq \emptyset\}\,.
    $$
    Since $T_{p}^{l*} g$ is always supported in $3I(p)$, we have that 
    $$
        \lvert \langle T_{\mathbf{A}}^{l*} g, T_{\mathbf{A}}^{l*} g \rangle \rvert \leq 2\sum_{p \in \mathbf{A}} \sum_{p' \in \mathbf{A}(p)} \lvert\langle T_p^{l*} g, T_{p'}^{l*} g\rangle\rvert\,.
    $$
    By Lemma \ref{lem tile TT*}, this is bounded up to constant factor by
    $$
        \delta^{-(2d + 2)\varepsilon_0} \sum_{p \in \mathbf{A}} \int \mathbf{1}_{E(p)}|g| \frac{1}{|I(p)|} \int |g| \sum_{p' \in \mathbf{A}(p)} \Delta(p, p')^{-1} \mathbf{1}_{E(p')}\,.
    $$
    Cubes $I(p')$ associated to $p' \in \mathbf{A}(p)$ are contained in $5I(p)$, so we may apply Hölder's inequality in the inner integral with $q < 2$ to bound this by a constant times
    $$
        \delta^{-(2d + 2)\varepsilon_0} \sum_{p \in \mathbf{A}} \int \mathbf{1}_{E(p)}|g| M^q |g|\frac{\|\sum_{p' \in \mathbf{A}(p)} \Delta(p, p')^{-1} \mathbf{1}_{E(p')}\|_{L^{q'}}}{|I(p)|^{1/q'}}\,,
    $$
    where $M^q g = M(g^q)^{1/q}$ is the $q$-maximal function.
    Using Lemma \ref{lem ant dens} to estimate the $L^{q'}$-norm, we estimate this by a constant times
    \begin{multline*}
        \delta^{\frac{\kappa}{q'} -(2d + 2)\varepsilon_0} \sum_{p \in \mathbf{A}} \int \mathbf{1}_{E(p)}|g| M^q |g|\\
        \leq \delta^{\frac{\kappa}{q'} -(2d + 2)\varepsilon_0}  \int |g| M^q|g| \lesssim \delta^{\frac{\kappa}{q'} -(2d + 2)\varepsilon_0} \|g\|_2^2\,.
    \end{multline*}
    Here we used disjointness of the sets $E(p)$ and $L^2$-boundedness of $M^q$ for $q < 2$.
    This completes the proof.
\end{proof}

\subsection{The high frequency part}
\label{sub sec high antichain}
Here we estimate the high frequency part $T_\mathbf{A}^h$.

We start by discretizing modulation frequencies. For this, we let $\varepsilon_1 > 0$ be a small positive number, much smaller than $\varepsilon_0$, it will be chosen at the end of this section. We fix finite subsets $M(\omega) \subset \omega$ of each dyadic frequency cube $\omega$, such that the following holds with $\rho = \delta^{\varepsilon_1}$:
\begin{enumerate}
    \item[(i)] $|M(\omega)| \lesssim \rho^{-\dim_h V}$,
    \item[(ii)] For each $N \in \omega$, there exists $N' \in M(\omega)$ with 
    \begin{equation}
        \label{eq par close}
        d_{\text{par}}(N, N') \leq \rho 2^{-s(\omega)}\,,
    \end{equation}
    where $d_{\text{par}}$ denotes the parabolic distance
    $$
        d_{\text{par}}(N, N') = \max\{\max_{1 \leq i \leq d} |N_i - N_i'|, |N_{d+1} - N_{d+1}'|^{1/2}\}\,.
    $$
    \item[(iii)] If $\omega = a + b \omega'$ then $M(\omega) = a + bM(\omega')$.
\end{enumerate}
For each tile $p$, we partition the set $E(p)$ into $E(p^c)$, $c \in M(\omega(p))$, such that for all $x \in E(p^c)$ 
\begin{equation}
    \label{eq par close 2}
    d_{\text{par}}(N(x), c) \leq \delta^{\varepsilon_1} 2^{-s(p)}\,.
\end{equation}
We define for a tile $p$ and $c \in M(\omega(p))$ the operator
$$
    T_{p}^{h, c} f(x) := \mathbf{1}_{E(p^c)} \int f(x-y)e^{ic \cdot y} \, \mathrm{d}\mu_{s(p)}^h(y)\,,
$$
and we define also
$$
    T_{\mathbf{A}}^{h,d} = \sum_{p \in \mathbf{A}} \sum_{c \in M(\omega(p))} T_p^{h,c}\,.
$$
By \eqref{eq par close 2}, we have for $x \in E(p^c)$ and all $y$ in the support of $\mu_{s(p)}^h$ that
$$
    |c \cdot y - N(x) \cdot y | \leq \delta^{\varepsilon_1}\,,
$$
hence
$$
    |T_p f(x) - \sum_{c \in M(\omega(p))} T_{p}^{h,c} f(x)| \leq \delta^{\varepsilon_1} \mathbf{1}_{E(p)}(x) \int |f(x-y)| \,\mathrm{d} \lvert \mu_{s(p)}^h \rvert(y)\,.
$$
Summing this over all $p \in \mathbf{A}$ and using that the corresponding sets $E(p)$ are disjoint, it follows that $|T_\mathbf{A}^h - T_{\mathbf{A}}^{h,d}|$ is dominated by $\delta^{\varepsilon_1}$ times a mollified maximal average along the parabola. Thus
\begin{equation}
\label{eq error}
    \| T_{\mathbf{A}}^h - T_\mathbf{A}^{h,d}\|_{2\to2} \lesssim \delta^{\varepsilon_1}\,.
\end{equation}
It now remains to estimate the discretized operator $T_{\mathbf{A}}^{h,d}$. By disjointness of the sets $E(p^c)$ and the third property of the sets $M(\omega)$, we have 
\begin{equation}
    \label{eq squaref dom}
    T_{\mathbf{A}}^{h,d} f(x) \leq S f(x)\,,
\end{equation}
where the square function $Sf$ is defined depending on $V$ as follows: In the case $V = \R^{d-1} \times \{0\}^2$ it is defined by
\begin{equation}
    \label{eq squaref 1}
    (S f(x))^2 = \sum_{s \in \mathbb{Z}} \sum_{N \in \mathbb{Z}^{d-1}\times\{0\}^2} \sum_{c \in M(\omega_0)} |f * (e^{i\delta_s(c + N) \cdot y} \mu_s^h(y))(x)|^2
\end{equation}
with $\omega_0 = [0,1)^{d-1} \times \{0\}^2$, and in the case $V = \{0\}^d \times \R$ it is defined by
\begin{equation}
    \label{eq squaref 2}
    (S f(x))^2 = \sum_{s\in\mathbb{Z}} \sum_{N \in \{0\}^d\times\mathbb{Z}} \sum_{c \in M(\omega_0)} |f * (e^{i\delta_s(c + N) \cdot y} \mu_s^h(y))(x)|^2
\end{equation}
with $\omega_0 = \{0\}^d \times [0,1)$.

\begin{lemma}
    \label{lem squaref}
    There exist $C >0$ such that the following holds.
    Let $S$ be defined by \eqref{eq squaref 1}  or by \eqref{eq squaref 2}. Then we have
    $$
        \|S\|_{2 \to 2} \leq C \delta^{\frac{2}{d+2} \varepsilon_0 - \frac{\dim_h V}{2}\varepsilon_1}\,.
    $$
\end{lemma}

\begin{proof}
    We first treat the case \eqref{eq squaref 1}.
    We have 
    $$
        \hat \mu_0(\xi) = \int_{1/8 < |y| < 1/3} e^{i (\xi' \cdot y + \xi_{d+1}|y|^2)} K_0(y) \, \mathrm{d}y\,.
    $$
    Thus we obtain from the method of stationary phase (see e.g. \cite{H90}, Lemma 7.7.3) and \eqref{eq ks size}:
    \begin{equation}
        \label{eq stat phase}
        |\hat{\mu}_0(\xi)|^2 \lesssim \min\{1, |\xi_{d+1}|^{-d}\}\,.
    \end{equation}
    Combining this with \eqref{eq smoothen def 1} yields
    \begin{equation}
    \label{eq muoh1}
        |\hat{\mu}_0^h(\xi)|^2 = |(1 - \hat \varphi^1(\delta^{2\varepsilon_0}\xi_{d+1})) \hat \mu_0(\xi)|^2 \lesssim \min\{ |\delta^{2\varepsilon_0} \xi_{d+1}|^2, |\xi_{d+1}|^{-d}\}\,. 
    \end{equation}
    Summing up \eqref{eq muoh1} gives
    \begin{equation}
    \label{eq sumh1}
        \sum_{\substack{N \in \mathbb{Z}^{d-1}\times\{0\}^2\\ |\xi' + c' + N'| \leq |\xi_{d+1}|}} |\hat \mu_0^h(\xi + c + N)|^2 \lesssim |\xi_{d+1}|^{d-1} \min\{ |\delta^{2\varepsilon_0} \xi_{d+1}|^2, |\xi_{d+1}|^{-d}\}\,.
    \end{equation}
    If $|\xi'| > |\xi_{d+1}|$ then we have 
    $$
        |\nabla(\xi' \cdot y + \xi_{d+1}|y|^2)| = |\xi' + \xi_{d+1} y| > \frac{1}{2}|\xi'|\,.
    $$
    Integrating by parts (see e.g. \cite{H90}, Theorem 7.7.1), we obtain 
    \begin{equation}
        \label{eq intparts}
        |\hat \mu_0(\xi)|^2 \lesssim \min\{1, |\xi'|^{-d-1}\}\,,
    \end{equation}
    and hence 
    \begin{equation}
    \label{eq muoh2}
        |\hat{\mu}_0^h(\xi)|^2 \lesssim  \min\{ |\delta^{2\varepsilon_0}\xi_{d+1}|  \min\{ 1,  |\xi'|^{-d-1}\}, |\xi'|^{-d-1}\}\,.
    \end{equation}
    Summing up estimate \eqref{eq muoh2} gives
    \begin{equation}
    \label{eq sumh2}
        \sum_{\substack{N \in \mathbb{Z}^{d-1}\times\{0\}^2\\ |\xi' + c' + N'| > |\xi_{d+1}|}} |\hat \mu_0^h(\xi + c + N)|^2 \lesssim \min\{ |\delta^{2\varepsilon_0} \xi_{d+1}|^2, |\xi_{d+1}|^{-1}\}\,.
    \end{equation}
    Combining \eqref{eq sumh1} and \eqref{eq sumh2} and using $\min\{a,b\} \leq \min\{ (a^3b^{d-1})^{1/(d+2)}, b\}$ yields
    $$
        \sum_{N \in \mathbb{Z}^{d-1}\times\{0\}^2} |\hat \mu_0^h(\xi + c + N)|^2 \lesssim \min\{\delta^{2\varepsilon_0/d} |\xi_{d+1}|^2,\, |\xi_{d+1}|^{-1}\}\,.
    $$
    Using dilation invariance of all assumptions on $K$ this implies
    $$
        \sum_{N \in \mathbb{Z}^{d-1}\times\{0\}^2} |\hat \mu_s^h(\xi + \delta_{-s}(c + N))|^2 \lesssim \min\{\delta^{\frac{12}{d+2}\varepsilon_0} |2^{2s}\xi_{d+1}|^2,\, |2^{2s}\xi_{d+1}|^{-1}\}\,.
    $$
    Summing in $s$, and also summing in the $\lesssim \delta^{-\varepsilon_1 \dim_h V}$ choices of $c$, we obtain 
    $$
        \sum_{s\in \mathbb{Z}} \sum_{N \in \mathbb{Z}^{d-1}\times\{0\}^2} \sum_{c \in M(\omega_0)}  |\hat \mu_s^h(\xi + \delta_{-s}(c + N))|^2 \lesssim \delta^{\frac{4}{d+2}\varepsilon_0 - \varepsilon_1 \dim_h V}\,.
    $$
    This completes the proof in the case \eqref{eq squaref 1} by Plancherel.
    
    Now we deal with the vertical modulation case \eqref{eq squaref 2}. The stationary phase estimate \eqref{eq stat phase} combined with \eqref{eq smoothen def 2} yields
    $$
        |\hat \mu_0^h(\xi)|^2 \lesssim \min\{1, |\xi_{d+1}|^{-d}\}  |\delta^{\varepsilon_0} \xi'|^2\,,
    $$
    and hence
    $$
        \sum_{N \in \{0\}^d \times \mathbb{Z}} |\hat \mu_0^h(\xi + c + N)|^2 \lesssim  |\delta^{\varepsilon_0} \xi'|^2\,.
    $$
    Combining the estimates \eqref{eq intparts} in the range $|\xi'| > |(\xi + c + N)_{d+1}|$ and \eqref{eq stat phase} in the range $|\xi'| \leq |(\xi + c + N)_{d+1}|$ gives 
    $$
        \sum_{N \in \{0\}^d \times \mathbb{Z}} |\hat \mu_0^h(\xi + c + N)|^2 \lesssim  |\xi'|^{1-d}\,.
    $$
    Therefore we have 
    $$
        \sum_{N \in \{0\}^d \times \mathbb{Z}} |\hat \mu_0^h(\xi + c + N)|^2 \lesssim \min\{\delta^{2\varepsilon_0} |\xi'|^{2}, |\xi'|^{1-d}\}\,.
    $$
    As before, this implies the square function estimate (with better dependence on $\delta$) using dilation invariance of all assumptions to sum in $s$, the bound on the number of summands to sum in $c$, and then Plancherel.
\end{proof}

Combining the estimate for $T_\mathbf{A}^l$ in Lemma \ref{lem l antichain}, for $T_{\mathbf{A}}^h -T_\mathbf{A}^{h,d}$ in \eqref{eq error} and for $T_\mathbf{A}^{h,d}$ in \eqref{eq squaref dom} and Lemma \ref{lem squaref}, and choosing $\varepsilon_0$ and then $\varepsilon_1$ sufficiently small, we obtain Proposition \ref{prop antichain}.

\section{Trees: Proof of Proposition \ref{prop tree}}
\label{sec tree}

Here we prove a mild generalization of Proposition \ref{prop tree}. We will estimate what we call generalized tree operators, with general kernels satisfying conditions described below. We need the estimates in Section \ref{sec forest} both for the singular Radon transform and for a version of it with a smoothened kernel.

\subsection{General setup}
\label{sub sec setup tree}
Let in the following $\mu_s, \lambda_s$, $s \in \mathbb{Z}$ be any finite measures satisfying the following conditions:
\begin{enumerate}
    \item[i)]  $\mu_s$, $\lambda_s$ are supported in $\delta_s(B(0,1/2))$,
    \item[ii)] $|\mu_s| \leq \lambda_s$,
    \item[iii)] $\mu_s(\R^{d+1}) = 0$,
    \item[iv)] we have
        \begin{equation}
        \label{eq fourier decay}
        \max\{|\widehat{D_{-s} \lambda_s}(\xi))|, |\widehat{ D_{-s} \mu_s}(\xi)|\} \leq |\xi|^{-1}\,,
        \end{equation}
        where $D_{-s} \mu_s := 2^{-(d+2)s} \mu_s \circ \delta_{-s}$,
    \item[v)] and there exist $\gamma > 0$ and $p, q < 2$ such that convolution with $D_{-s} \lambda_s$ and $D_{-s} \mu_s$ is bounded from $L^{p}$ to $L^{q' + \gamma}$ with norm at most $1$ for each $s$.
\end{enumerate}

Relevant for the proof of Theorem \ref{main Lp} are the measures defined by 
\begin{equation}
    \label{eq mus def}
    \int f \, \mathrm{d}\mu_s = \int_{\R^d} f(y,|y|^2) K_s(y) \, \mathrm{d}y\,,
\end{equation}
\begin{equation}
    \label{eq las def}
    \int f \, \mathrm{d}\lambda_s = \int_{\R^d} f(y,|y|^2) 2^{-ds}\eta(2^{-s} y) \, \mathrm{d}y\,,
\end{equation}
and smoothened versions thereof. They clearly satisfy conditions i) to iii), condition iv) follows from standard stationary phase estimates (e.g. \cite{H90}, Lemma 7.7.3) and for condition v) one can take any $(p,q)$ such that $(\frac{1}{p}, \frac{1}{q})$ is in the interior of the convex hull of $(1,0), (0,1), (\frac{d+1}{d+2}, \frac{d+1}{d+2})$ (see \cite{Littman1973}).

\subsection{Generalized tree operators}
\label{sub sec intro tree}
With the choice of $\mu_s$ from \eqref{eq mus def}, the tile operators defined in Section \ref{sec discretization} can be written as
$$
    T_p f(x) = \mathbf{1}_{E(p)} (x) \int f(x-y)e^{iN(x) \cdot y} \mathrm{d}\mu_s(y)\,,
$$
and tree operators can be written as
$$
    T_{\mathbf{T}} f(x) = \sum_{p \in \mathbf{T}} T_p f(x) = \sum_{s \in \sigma(x)} \int f(x-y)e^{iN(x) \cdot y} \mathrm{d}\mu_s(y)\,,
$$
where $\sigma(x) := \{s(p) \, : \, p \in \mathbf{T}, x \in E(p)\}$.
Since $\mathbf{T}$ is a subset of the set of all admissible tiles, the set $\sigma(x)$ is always contained in 
$$
    J = J(\mathbf{T}) := \{s \in \mathbb{Z} \, : \, \exists \omega \ \text{admissible,} \,  s(\omega) = s,\, \omega(\mathbf{T}) \subset \omega\}\,,
$$
and convexity of the tree $\mathbf{T}$ implies that there are $\underline{\sigma}(x) \leq \overline{\sigma}(x)$ with 
$$
    \sigma(x) = J \cap [\underline{\sigma}(x), \overline{\sigma}(x)]\,.
$$

Motivated by this, we define a generalized tree to be a pair $(\mathbf{T}, \sigma')$, where $\mathbf{T}$ is a tree and $\sigma'$ is a function associating to each $x$ a set
$$\sigma'(x)  = J \cap [\underline{\sigma}'(x), \overline{\sigma}'(x)] \subset \sigma(x)\,.$$
The generalized tree operator associated to $(\mathbf{T}, \sigma')$ is defined by
$$
    T_{\mathbf{T},\sigma'}f(x) = \sum_{s \in \sigma'(x)} \int f(x-y)e^{iN(x) \cdot y} \mathrm{d}\mu_s(y)\,.
$$
The above discussion shows that this includes in particular the tree operators defined in Section \ref{sec outline}, by choosing $\mu_s$ as in \eqref{eq mus def} and $\sigma'(x) = \sigma(x)$ for each $x$. 

\subsection{Single tree estimate}
\label{sub sec sing tree}

We now dominate the generalized tree operators $T_{\mathbf{T},\sigma'}$ by a sum of two simpler operators $T_{*\mathbf{T},\sigma'}$ and $M_{\mathbf{T}}$, resembling a maximally truncated singular integral and a maximal average along the paraboloid.

Define the unmodulated operator associated to $(\mathbf{T},\sigma')$ by
\begin{align*}
    T_{*\mathbf{T},\sigma'} f(x) := \sum_{s \in \sigma'(x)} f * \mu_s(x) = \sum_{s = \underline{\sigma}'(x)}^{\overline{\sigma}'(x)} f*(\mathbf{1}_J(s) \mu_s)(x)\,.
\end{align*}
This is a maximally truncated singular integral along the paraboloid, and it follows from standard square function arguments as in \cite{Stein+1978} that $T_{*\mathbf{T}}$ is bounded on $L^2$.

Define also the maximal average $M_{\mathbf{T}}$ associated to $\mathbf{T}$:
$$
    M_{\mathbf{T}} f(x) := \sup_{s \in \sigma(x)} |f| * \lambda_s(x)\,.
$$
The operator $M_{\mathbf{T}}$ is bounded above pointwise by the maximal average associated to $(\lambda_s)_{s\in \mathbb{Z}}$, which is bounded on $L^2$ under our assumptions.

We have the following pointwise estimate for $T_{\mathbf{T},\sigma'}$ in terms of $T_{*\mathbf{T},\sigma'}$ and $M_{\mathbf{T}}$. 

\begin{lemma}
    \label{lem tree reduction}
    There exists $C > 0$ such that for each generalized tree $(\mathbf{T},\sigma')$ 
    $$
        |T_{\mathbf{T},\sigma'} f(x)| \leq |T_{*\mathbf{T},\sigma'}(e^{-iN(\mathbf{T})(y)} f(y))(x)| + C M_{\mathbf{T}} f(x)\,,
    $$
    where $N(\mathbf{T})$ is the smallest element, in lexicographic order, of $\omega(\mathbf{T})$. 
\end{lemma}

\begin{proof}
    Suppose that $p \in \mathbf{T}$ with $s(p) = \max \sigma(x) - k$. Then there exists some $p' \in \mathbf{T}$ with $s(p') = s(p) + k$ and $x \in E(p')$. Since $N(x) \in \omega(p')$ and $N(\mathbf{T}) \in \omega(p')$,  it follows that
    $$
        |(N(\mathbf{T}) - N(x)) \cdot y| \leq d 2^{-k}
    $$
    for all $y$ in the support of $\mu_{s(p)}$.
    Hence we have 
    \begin{multline*}
        \left|T_{p}f(x)
        -e^{iN(\mathbf{T})\cdot x} \mathbf{1}_{E(p)}(x) \int f(x - y)e^{-iN(\mathbf{T})\cdot (x - y)} \mathrm{d}\mu_{s(p)}(y)\right|\\
        = \left|\mathbf{1}_{E(p)}(x) \int f(x-y)(e^{iN(x)\cdot y} - e^{iN(\mathbf{T})\cdot y})\mathrm{d}\mu_{s(p)}(y)\right|\\
        \lesssim 2^{-k} \mathbf{1}_{E(p)}(x) |f|*\lambda_s(x)\,,
    \end{multline*}
    using that $|\mu_s| \leq \lambda_s$ and $|e^{iN(x)\cdot y} - e^{iN(\mathbf{T})\cdot y}| \lesssim 2^{-k}$ on the support of $\mu_{s(p)}$.
    Summing over all tiles $p \in \mathbf{T}$, we obtain:
    \begin{multline*}
        |T_{\mathbf{T},\sigma'} f(x) - e^{N(\mathbf{T})\cdot x} T_{*\mathbf{T},\sigma'}(e^{-N(\mathbf{T})\cdot y} f(y))(x)|\\
        \leq \sum_{p \in \mathbf{T}} \left| T_{p}f(x)
        -e^{iN(\mathbf{T})\cdot x} \mathbf{1}_{E(p)}(x) \int f(x - y)e^{-iN(\mathbf{T})\cdot (x - y)} \mathrm{d}\mu_{s(p)}(y)\right|\\
        \lesssim \sum_{k \geq 0} 2^{-k} \sum_{\substack{p \in \mathbf{T}\\ s(p) = \max\sigma(x) - k}} \mathbf{1}_{E(p)}(x) |f|*\lambda_s(x)\,.
    \end{multline*}
    By disjointness of the sets $E(p)$ for tiles $p$ of a fixed scale, the inner sum is bounded by $M_\mathbf{T}$. This completes the proof.
\end{proof}

\subsection{Low density trees and sparse bounds}
\label{sub sec sparse}

Lemma \ref{lem tree reduction} implies that the operators $T_{\mathbf{T},\sigma'}$ are bounded on $L^2$ uniformly over all generalized trees $(\mathbf{T},\sigma')$. We will now improve this estimate for trees of small density $\delta$, using sparse bounds for the operators $T_{*\mathbf{T},\sigma'}$ and $M_{\mathbf{T}}$.
These sparse bounds are variants of bounds for prototypical singular Radon transforms that were shown by Oberlin \cite{Oberlin2019}.
Our setting is slightly more general than in \cite{Oberlin2019}, and we need a more precise estimate, but it still follows from Oberlin's proof with only minor modifications. 

A collection $\mathcal{S}$ of dyadic cubes is called sparse if for each $Q \in \mathcal{S}$ there exists a subset $U(Q) \subset Q$ with $|U(Q)| \geq |Q|/2$, such that the sets $U(Q)$ are pairwise disjoint.
For every cube $Q$ and $p \geq 1$, we denote by
$$
    \langle f \rangle_{Q,p} := \left(\frac{1}{|Q|} \int_Q |f|^p \, \mathrm{d}x \right)^{1/p}
$$
the $p$-average of a function $f$ over $Q$. Finally, given a dyadic cube $Q$ and a tree $\mathbf{T}$, we define 
$$
    E(Q) = E(\mathbf{T}, Q) := \bigcup_{\substack{p \in \mathbf{T}\\ I(p) \subset 3Q}} E(p)\,.
$$

\begin{lemma}[\cite{Oberlin2019}]
    \label{lem sparse}
    Suppose that $(p,q)$ are as in condition v) in Subsection \ref{sub sec setup tree}.
    Then there exists a constant $C > 0$ such that for every generalized tree $(\mathbf{T},\sigma')$ and all $f, g$ there exists a sparse collection of cubes $\mathcal{S} \subset \mathbf{D}$ such that 
    \begin{equation}
        \label{eq sparse goal}
        \left\lvert\int T_{*\mathbf{T},\sigma'} f(x) g(x) \, \mathrm{d}x \right\rvert \leq C \sum_{Q \in \mathcal{S}} |Q| \langle f \rangle_{Q, p} \langle \mathbf{1}_{E(Q)} g  \rangle_{3Q, q}\,.
    \end{equation}
    The same statement holds with $T_{*\mathbf{T},\sigma'}$ replaced by $M_{\mathbf{T}}$.
\end{lemma}

\begin{proof}
    This follows from the proof of Theorems 1.3 and 1.4 in \cite{Oberlin2019} (note that our $q$ is his $q'$), with the following modifications. 
    Firstly, Oberlin considers operators 
    $$
        \sum_{s \in \mathbb{Z}} \varepsilon_s(x) f * \mu_s\,,
    $$
    with $\varepsilon_s(x) \in [-1,1]$, where $\mu_s$ is an isotropic dilate of a fixed measure $\mu_0$. In our setting the $\mu_s$ are \emph{anisotropic} dilates of $D_{-s} \mu_s$, but Oberlin's proof goes through in the anisotropic setting as well. Furthermore, in our setting the measures $D_{-s} \mu_s$ are not identical, they depend on $s$. However, they satisfy all assumptions of Oberlin's theorems uniformly in $s$, by i) - v) above, and the proof still goes through with this assumption.

    It remains to explain why we can insert $\mathbf{1}_{E(Q)}$ in the $q$-average over $3Q$ in \eqref{eq sparse goal}. We explain it for $T_{*\mathbf{T}, \sigma'}$, the argument for $M_\mathbf{T}$ is very similar. 
    Oberlin constructs the sparse collection $\mathcal{S}$ by an iterative argument starting from a large cube $Q_0$, such that $f$, $g$ are supported in $Q_0$, $3Q_0$ respectively. The expression on the left hand side of \eqref{eq sparse goal} does not change if $f$ is restricted to $I(\mathbf{T})$ and $g$ to $3I(\mathbf{T})$, so we can choose $Q_0 = I(\mathbf{T})$. Oberlin then defines operators $T_Q$, which in our notation are
    \begin{equation}
        \label{eq TQ def}
        T_Q f(x) = \sum_{\substack{p \in \mathbf{T}\\s(p) \leq s(Q)}} \mathbf{1}_{E(p)}(x) \mathbf{1}_{s(p) \in \sigma'(x)} \cdot \mu_{s(p)} * (\mathbf{1}_Qf)(x)\,,
    \end{equation}
    so that $T_{*\mathbf{T},\sigma'} = T_{Q_0}$. Note that in the sum in \eqref{eq TQ def} only tiles $p$ with $Q \cap 3I(p) \neq \emptyset$ contribute. But if $s(p) \leq s(Q)$ and $Q \cap 3I(p) \neq \emptyset$, then $I(p) \subset 3Q$. Thus, $T_Q = \mathbf{1}_{E(Q)} T_Q$, and hence 
    \begin{equation}
    \label{eq sparse bil}
        \langle T_Q f, g\rangle = \langle T_Q f, \mathbf{1}_{E(Q)} g\rangle\,.
    \end{equation}
    Our goal is to estimate 
    $$
        |\langle T_{Q_0} f, g\rangle| = |\langle T_{Q_0} f, \mathbf{1}_{E(Q_0)} g\rangle|\,.
    $$
    Oberlin shows that for every dyadic cube $Q$ and every $f$ there exists a collection $Q_1(Q)$ of dyadic cubes such that
    \begin{equation}
    \label{eq oberlin 2}
        |\langle T_Q f, g\rangle| \leq |Q| \langle f\rangle_{Q, p} \langle g\rangle_{3Q, q} + \sum_{\substack{Q' \in Q_1(Q)\\Q' \subset Q}} \langle T_{Q'} f, g\rangle\,.
    \end{equation}
    This follows from his equations (3.2) and (3.3) and the claims below them. The sparse collection $\mathcal{S}$ is then constructed by starting with $Q_0$ and iteratively adding for all $Q \in \mathcal{S}$ the cubes $Q_1(Q)$ to $\mathcal{S}$. Combining \eqref{eq sparse bil} and \eqref{eq oberlin 2} one obtains \eqref{eq sparse goal} for this $\mathcal{S}$, and Oberlin shows that $\mathcal{S}$ is sparse. This completes the proof. Note that the only change in our argument compared to Oberlin's is that we use \eqref{eq sparse bil} to insert $\mathbf{1}_{E(Q)}$ into the $q$-averages.
\end{proof}

\begin{corollary}
    \label{cor sparse}
    Let $q$ be as in v) in Subsection \ref{sub sec setup tree}.
    Then for each $\varepsilon < \frac{1}{q} - \frac{1}{2}$ there exists $C > 0$ such that for each generalized tree $(\mathbf{T},\sigma')$ with $\mathbf{T}$ of density $\delta$, we have
    $$
        \|T_{*\mathbf{T},\sigma'}\|_{2 \to2} \leq C\delta^{\varepsilon}\,.
    $$
    The same statement holds with $T_{*\mathbf{T},\sigma'}$ replaced by $M_\mathbf{T}$.
\end{corollary}

\begin{proof}
    Let $\mathbf{L} = \mathbf{L}(\mathbf{T})$ be the collection of maximal dyadic cubes $L$ such that there exists some $p \in \mathbf{T}$ with $L \subset I(p)$ but there exists no $p \in \mathbf{T}$ with $I(p) \subset L$. This is a partition of $I(\mathbf{T})$. Define 
    $$
        E(L) = L \cap \bigcup_{p' \in \mathbf{T}} E(p')\,.
    $$
    We claim that for each $L \in \mathbf{L}$ we have 
    \begin{equation}
        \label{eq delta claim}
        |E(L)| \leq \delta |L|\,.
    \end{equation}
    To prove this, fix $L \in \mathbf{L}$.
    By definition of $\mathbf{L}$ there exists $p \in \mathbf{T}$ with $I(p) \subset \hat L$. Define $p_L$ to be $p$ if $I(p) = \hat L$, and else let $p_L$ be the unique tile with $I(p_L) = \hat L$ and $\omega(\mathbf{T}) \subset \omega(p_L)$. If $p' \in \mathbf{T}$ with $I(p') \cap L \neq \emptyset$ then we have $p_L \leq p'$. Thus 
    $$
        |E(L)| \leq |E(p_L)| \leq \dens(\mathbf{T}) |L| \leq \delta |L|\,,
    $$
    giving the claim \eqref{eq delta claim}. Since $\mathbf{L}$ forms a partition of $I(\mathbf{T})$, we have for all $Q \in \mathbf{D}$
    $$
        |E(Q)| \leq \sum_{L \subset 3Q} |E(L)| \leq \delta \sum_{L \subset 3Q} |L| \leq \delta 3^{d+2} |Q|\,.
    $$
    Thus, for every $\tilde q > q$, we have by Hölder's inequality
    \begin{equation}
    \label{eq qq}
        \langle \mathbf{1}_{E(Q)} g \rangle_{3Q ,q} \leq (3^{d+2}\delta)^{1/q - 1/\tilde q} \langle g \rangle_{3Q ,\tilde q} \,.
    \end{equation}
    Now pick $(p,q)$ as in v), and pick $\tilde q$ with $q < \tilde q < 2$. We obtain from Lemma \ref{lem sparse} and \eqref{eq qq}
    \begin{align*}
        \left\lvert\int T_{*\mathbf{T},\sigma'} f(x) g(x) \, \mathrm{d}x \right\rvert
        &\lesssim \delta^{1/q - 1/\tilde q} \sum_{Q \in \mathcal{S}} |Q| \langle f \rangle_{Q, p} \langle g \rangle_{3Q, \tilde q}\\
        &\leq 2 \delta^{1/q - 1/\tilde q}  \sum_{Q \in \mathcal{S}} |U(Q)| \langle f \rangle_{Q, p} \langle g \rangle_{3Q, \tilde q}\\
        &\leq 2 \delta^{1/q - 1/\tilde q}  \sum_{Q \in \mathcal{S}} \int_{U(Q)} M_p f M_{\tilde q} g\\
        &\lesssim \delta^{1/q - 1/\tilde q}  \|f\|_2 \|g\|_2\,.
    \end{align*}
    Here we used boundedness of the $p$- and $\tilde q$- maximal function on $L^2$ for $p, \tilde q < 2$. This yields the  desired result for $T_{*\mathbf{T}, \sigma'}$. 
    
    For $M_{\mathbf{T}}$ the proof is exactly the same.  
\end{proof}

Combining Corollary \ref{cor sparse} with Lemma \ref{lem tree reduction} we obtain an estimate for a single tree with decay in the density parameter.

\begin{corollary}
    \label{cor single tree}
    For each $\varepsilon < \frac{1}{q} - \frac{1}{2}$ there exist $C > 0$ such for each generalized tree $(\mathbf{T},\sigma')$ with $\mathbf{T}$ of density $\delta$, we have
    $$
        \|T_{\mathbf{T},\sigma'}\|_{2 \to 2} \leq C \delta^{\varepsilon}\,.
    $$
\end{corollary}

Recall from Subsection \ref{sub sec setup tree}  that for the measures $\mu_s$ defined by \eqref{eq mus def}, condition v) holds for all $p,q < 2$ such that $(1/p, 1/q)$ is in the interior of the convex hull of $(1,0), (0,1), (\frac{d+1}{d+2}, \frac{d+1}{d+2})$.
In this case the condition in Corollary \ref{cor single tree} becomes $\varepsilon < \frac{1}{2} - \frac{1}{2(d+1)}$.

\section{Forests: Proof of Proposition \ref{prop forest}}
\label{sec forest}

To prove Proposition \ref{prop forest}, it remains to combine the bounds for the operators $T_\mathbf{T}$ from Proposition \ref{prop tree} for all trees in a forest $\mathbf{F}$. For this, we will show almost orthogonality estimates for tree operators associated to separated trees.
\subsection{Basic orthogonality estimates}
\label{sub sec bas sep tree}

As a first step, we show that the adjoint generalized tree operator $T_{\mathbf{T}, \sigma'}^*$ is frequency localized near $N_\mathbf{T}$. We use the definitions from Subsection \ref{sub sec intro tree}.

We let $\Psi$ denote the set of Schwartz functions $\psi$ such that 
\begin{enumerate}
    \item[(i)] $\psi$ is supported in $B(0,1)$,
    \item[(ii)]  $\psi$ has integral $1$,
    \item[(iii)] $|\partial^\alpha \psi(x)| \leq L$ for all $|\alpha| \leq 10d$,
\end{enumerate}
where $L$ is chosen sufficiently large such that for each $\xi$ with $|\xi| \geq 1$, there exists $\psi \in \Psi$ with $\hat \psi(\xi) = 0$.

For $\psi \in \Psi$, we define the frequency projection $\Pi_{R, N} = \Pi_{R, N , \psi}$ by
$$
    \widehat{\Pi_{R, N} f}(\xi) = (1 - \hat \psi(R^{-1}(\xi - N))) \hat f(\xi)\,.
$$

\begin{lemma}
    \label{lem T* proj}
    For each $\kappa < \frac{1}{2(d+1)}$ there exist $C$ such that the following holds. For each generalized tree $(\mathbf{T},\sigma')$ such that each tile in $\mathbf{T}$ has scale at least $0$, all $R \geq 1$ and all $\psi \in \Psi$
    $$
        \|T_{\mathbf{T}, \sigma'} \Pi_{R,N(\mathbf{T})}\|_{2 \to2} \leq C R^{-\kappa}\,. 
    $$
\end{lemma}

\begin{proof}
    We start by separating the $\log_2(R)$ largest scales:
    \begin{align}
        \label{eq large scales}
        T_{\mathbf{T},\sigma'} f(x) &= \sum_{\substack{s\in\sigma'(x)\\ s \geq \overline{\sigma}'(x) - \log_2(R)}} \int f (x-y) e^{iN(x)\cdot y}\mathrm{d}\mu_s(y) \\
        \label{eq small scales}
        &\quad + \sum_{\substack{s\in\sigma'(x)\\ s < \overline{\sigma}'(x) - \log_2(R)}} \int f(x-y) e^{iN(x)\cdot y}\mathrm{d}\mu_s(y)\,.
    \end{align}
    The small scale contribution \eqref{eq small scales} is close to the $N(\mathbf{T})$-modulated operator:
    \begin{multline}
        \left\lvert\sum_{\substack{s\in\sigma'(x)\\ s < \overline{\sigma}'(x) -  \log_2(R)}} \int f(x-y) (e^{-iN(x) \cdot y} -e^{-iN(\mathbf{T}) \cdot y}) \, \mathrm{d}\mu_s(y) \right\rvert \\
        \lesssim \sum_{\substack{s\in\sigma'(x)\\ s < \overline{\sigma}'(x) - \log_2(R)}} 2^{s - \overline{\sigma}'(x)} \int |f(x-y)| \, \mathrm{d}|\mu_s|(y)
        \lesssim R^{-1} M^{\lambda} f\,,
        \label{eq max fct comp}
    \end{multline}
    where $M^{\lambda}$ denotes the maximal average associated to $\lambda$.
    Here we have used that if $s \in \sigma'(x)$ with $s = \overline \sigma'(x) - k$, then there exists a frequency cube of scale $s + k$ containing $N(\mathbf{T})$ and $N(x)$, which implies that for all $y$ in the support of $\mu_s$ we have $| (N(x) - N(\mathbf{T})) \cdot y| \lesssim 2^{-k}$.
    So to bound \eqref{eq small scales}, it suffices to show the following estimate for the $N(\mathbf{T})$-modulated operator:
    $$
        \left\|\sum_{\substack{s \in \sigma'(x)\\ s < \overline{\sigma}'(x)- \log_2(R)}} \int  e^{iN(\mathbf{T})\cdot (x-y)} (\Pi_{R,N(\mathbf{T})}f)(x-y) \, \mathrm{d}\mu_s(y)\right\|_2 \lesssim R^{-1/2} \|f\|_2\,.
    $$
    Replacing $f$ by $e^{iN(\mathbf{T})\cdot x} f(x)$ and taking
     differences it suffices to show
    $$
        \left\lVert\sup_{0 \leq \underline{\sigma}} \left\lvert  \sum_{\underline{\sigma} \leq s} \mathbf{1}_J(s) \int (\Pi_{R,0} f)(x-y) \, \mathrm{d}\mu_s(y) \right\rvert\right\rVert_{2} \lesssim R^{-1/2} \|f\|_2\,.
    $$
    This follows from a standard square function argument, using that by \eqref{eq fourier decay} we have for all $\underline{\sigma} \geq 0$
    $$
        |1- \hat \psi(R^{-1} \xi)|\sum_{\underline\sigma \leq s} |\hat \mu_s(\xi)| \lesssim  R^{-1/2} 2^{-\underline\sigma/2} \,.
    $$

    Now we treat the large scales \eqref{eq large scales}, using that only logarithmically many scales contribute at each point. We discretize modulation frequencies. Let $\gamma > 0$, to be chosen later. For each dyadic frequency cube $\omega$, we fix finite subsets $M(\omega)$ satisfying conditions (i), (ii), (iii) in Section \ref{sub sec high antichain}, with $\rho = R^{-\gamma}$. For each $s \in \mathbb{Z}$, let $\omega_s$ be the unique frequency cube of scale $s$ containing $N(\mathbf{T})$. For each $x$ and each $s \in \sigma'(x)$, we pick a frequency $c(x, s) \in M(\omega_s)$ such that $d_{\text{par}}(N(x), c(x,s)) \leq R^{-\gamma}2^{-s}$.
    By a similar computation as in \eqref{eq max fct comp}
    \begin{multline*}
        \left\lvert\sum_{\substack{s\in\sigma'(x)\\ s \geq \overline{\sigma}'(x) -  \log_2(R)}} \int f(x-y) (e^{-iN(x) \cdot y} - e^{ic(x, s) \cdot y}) \, \mathrm{d}\mu_s(y) \right\rvert\\
        \leq \log_2(R) R^{-\gamma} M^{\lambda}f\,.
    \end{multline*}
    Thus we may replace $N(x)$ by $c(x, s)$ in \eqref{eq large scales}. We bound the resulting sum by the number of nonzero summands, which is at most $\log_2(R)$ at each point $x$, times the maximal summand.
    The maximal summand is controlled by the estimate 
    $$
        \left\lVert\sup_{s \geq 0} \sup_{c \in M(s)} \left\lvert \int (\Pi_{R,N}f)(x-y) e^{i c \cdot y} \, \mathrm{d}\mu_s(y) \right\rvert\right\rVert_{2}
        \lesssim R^{\gamma \dim_h V } R^{-1/2}\|f\|_2\,.
    $$
    This follows from
    $
        |(1- \hat \psi(R^{-1} (\xi - N + c)))|\lvert \hat \mu_s(\xi) \rvert \lesssim R^{-1/2} 2^{-s/2}
    $
    and a standard square function argument, using that there are at most $R^{\gamma \dim_h V}$ choices of $c$ for each $s$.
    Optimizing $\gamma$, we obtain the lemma. 
\end{proof}

An immediate corollary are almost orthogonality estimate for separated trees with the same top cube.

\begin{corollary}
    \label{cor sep trees baby}
    For each $\kappa < \frac{1}{2d + 3}$ there exists $C>0$ such that the following holds. 
    Let $\mathbf{T}_1, \mathbf{T}_2$ be a pair of $\Delta$-separated, normal trees with $ I(\mathbf{T}_1) =  I(\mathbf{T}_2) =: I$. Let $(\mathbf{T}_1, \sigma_1), (\mathbf{T}_2, \sigma_2)$ be generalized tree operators, with possibly different $\mu_s$. Then 
    \begin{equation}
    \label{eq sep trees baby}
        \left\lvert \int T_{\mathbf{T}_1,\sigma_1}^* g_1 \overline{T_{\mathbf{T}_2,\sigma_2}^* g_2}\right\rvert \leq C \Delta^{-\kappa} \| g_1 \|_{L^2(I)}  \|  g_2 \|_{L^2(I)}\,.
    \end{equation}
\end{corollary}

\begin{proof}
    We drop the $\sigma_i$ from the notation and write $T_{\mathbf{T}_i, \sigma_i} = T_{\mathbf{T}_i}$.
    By scaling, we may assume that the minimal scale of a tile in $\mathbf{T}_1 \cup \mathbf{T}_2$ is $0$. Furthermore we may assume that $\Delta > 3$, otherwise \eqref{eq sep trees baby} follows from $L^2$ boundedness of generalized tree operators. Then it holds by $\Delta$-separation that $|N(\mathbf{T}_1) - N(\mathbf{T}_2)| \geq \Delta/3$.
    
    We let $\kappa' = \kappa/(1 - \kappa) < 1/(2(d+1))$ and $\gamma = 1/(1+\kappa')$, define 
    $\Pi_i = \Pi_{\Delta^\gamma, N(\mathbf{T}_i)}$ for some function $\psi \in \Psi$, and split up 
    \begin{multline*}
         \langle T_{\mathbf{T}_1}^* g_1, T_{\mathbf{T}_2}^* g_2\rangle = \langle \Pi_1 T_{\mathbf{T}_1}^* g_1, \Pi_2 T_{\mathbf{T}_2}^* g_2\rangle + \langle (1-\Pi_1) T_{\mathbf{T}_1}^* g_1, \Pi_2 T_{\mathbf{T}_2}^* g_2\rangle\\
        + \langle \Pi_1 T_{\mathbf{T}_1}^* g_1, (1 - \Pi_2) T_{\mathbf{T}_2}^* g_2\rangle + \langle (1-\Pi_1) T_{\mathbf{T}_1}^* g_1,(1 -\Pi_2) T_{\mathbf{T}_2}^* g_2\rangle\,.
    \end{multline*}
    The first three terms are bounded by $\Delta^{-\gamma \kappa'} \|g_1\|_{L^2(I)} \|g_2\|_{L^2(I)}$ by Lemma \ref{lem T* proj},
    and the last term is bounded by 
    $$
    \|(1 - \Pi_1)(1 - \Pi_2)\|_{2 \to 2} \|g_1\|_{L^2(I)} \|g_2\|_{L^2(I)} \lesssim \Delta^{\gamma - 1} \|g_1\|_{L^2(I)} \|g_2\|_{L^2(I)}\,,
    $$  
    because $\psi \in \Psi$ and $|N(\mathbf{T}_1) - N(\mathbf{T}_2)| \geq \Delta/3$. This gives the desired estimate since $\gamma \kappa' = \gamma - 1= \kappa$.
\end{proof}

\subsection{Auxiliary estimates for oscillatory integrals}
\label{sub sec aux sep tree}

In this subsection we prepare the proof of a version of Corollary \ref{cor sep trees baby} for a general pair of separated trees  by showing some estimates for oscillatory integrals along paraboloids.

From now on,  $\mu_s$ and $\lambda_s$ are fixed to be the measures defined in \eqref{eq mus def} and \eqref{eq las def}. We define
\begin{equation}
    \label{eq chi def}
    \chi_{s, \varepsilon} = \varepsilon^{-d-1} \mathbf{1}_{B(0,\varepsilon)} * (\lambda_{s-1} + \lambda_s + \lambda_{s+1})
\end{equation}
and $\chi_{s,\varepsilon,\delta} = \delta^{-1} \chi_{s, \varepsilon} \mathbf{1}_{|x_d| < \delta2^s}$. We also define the associated maximal convolution operators
\begin{equation}
    \label{eq M1 def}
    M^{\chi, 1} f(x) := \sup_{s \in \mathbb{Z}} \sup_{\varepsilon > 0} |f| * \tilde \chi_{s, \varepsilon}\,,
\end{equation}
\begin{equation}
    \label{eq M2 def}
    M^{\chi,2} f(x) := \sup_{s \in \mathbb{Z}} \sup_{\varepsilon > 0} \sup_{\delta > 0} |f| * \tilde \chi_{s, \varepsilon, \delta}\,,
\end{equation}
where $\tilde \chi(x) = \chi(-x)$.
Both $M^{\chi,1}$ and $M^{\chi,2}$ are bounded on $L^2$, because they are dominated by the composition of the Hardy-Littlewood maximal function and maximal averages along the parabolas in the direction of the coordinate axes $\{te_i + t^2 e_{d+1} \, : \, t \in \R\}$ for $i = 1, \dotsc, d$.

\begin{lemma}
    \label{lem sep tree aux}
    There exists $C > 0$ such that the following holds. For all $a,b > 0$ with $b \geq 4a$, all $N \in V$, all $\psi\in \Psi$ and all $s \geq -1$ we have
    \begin{equation}
        \label{eq osc aux 1}
        |(e^{iN \cdot y} \mu_s(y))*\psi_{a}(x)| \lesssim  b \chi_{s, a, b}(x) + \frac{1}{ab|N|}  \chi_{s, a}(x)\,.
    \end{equation}
    Furthermore, if $\varphi \geq 0$ has integral $1$  and $\mu_s^l = \mu_s * \varphi$, then
    \begin{equation}
        \label{eq osc aux 2}
        |(e^{iN \cdot y} \mu_s^l(y))*\psi_{a}(x)| \lesssim  b \chi_{s, a, b} * \varphi(x) + \frac{1}{ab|N|}  \chi_{s, a} * \varphi(x) \,.
    \end{equation}
    Here we write $\psi_t(x) = t^{-d-1} \psi(t^{-1}x)$.
\end{lemma}

\begin{proof}
    The estimate \eqref{eq osc aux 2} follows from \eqref{eq osc aux 1} and the triangle inequality.
    For $|x_d| < b2^s$ estimate \eqref{eq osc aux 1} also follows directly from the triangle inequality. For $|x_d| \geq b2^s$ we distinguish cases depending on $V$. 
    
    If $V = \R^{d-1} \times \{0\}^2$ then $N_{d} = N_{d+1} = 0$, and without loss of generality we have that $|N_{d-1}| \geq |N|/d$.
    We let $\tilde y = (y_1, \dotsc, y_{d-2})$ and put 
    $$
        y  = (\tilde y, \sqrt{h^2 - |\tilde y|^2} \cos(\theta), \sqrt{h^2 - |\tilde y|^2} \sin(\theta))\,.
    $$
    We abbreviate $r = \sqrt{h^2 - |\tilde y|^2}$ and change variables to obtain
    \begin{multline}
         \int \psi_{a}(x - z) e^{iN\cdot z} \, \mathrm{d}\mu_s(z)
         \\
        = \int \psi_{a}(x' - y, x_{d+1} - |y|^2) e^{iN(x)\cdot (y, |y|^2)} K_s(y) \, \mathrm{d}y \\
        = \int\!\!\int\!\!\int \Gamma_{h, \tilde y}(\theta) e^{i \phi_{h, \tilde y}(\theta)} \, \mathrm{d}\theta\, h \, \mathrm{d}h \, \mathrm{d}\tilde y\,, \label{eq ibp prep}
    \end{multline}
    where we put  $\phi_{h, \tilde y}(\theta) = N_{d-1} r \cos(\theta)$ and
    \begin{multline*}
        \Gamma_{h, \tilde y}(\theta) = \psi_{a}(\tilde x - \tilde y, x_{d-1} - r \cos(\theta), x_d - r \sin(\theta), x_{d+1} - h) \\
        K_s(\tilde y, r \cos(\theta), r \sin(\theta))e^{i \tilde N \cdot \tilde y}\,.
    \end{multline*}
    Using \eqref{eq ks size} and the assumption $s \geq -1$, we find that
    \begin{equation}
        \label{eq gamma der}
        \left|  \Gamma'_{h, \tilde y}(\theta)\right| \lesssim a^{-d-2} 2^{-ds} r\,.
    \end{equation}
    Since $|x_d| > b2^s$ and since $|x_d - r \sin(\theta)| \leq a$ on the support of $\Gamma_{h, \tilde y}$, we have additionally that $|\sin(\theta)| \geq b - 2^{-s}a \geq b/2$. Thus 
    \begin{equation}
        \label{eq phi der}
        \left|  \phi'_{h, \tilde y}(\theta) \right| \geq \frac{|N|br}{2d}\qquad \text{and} \qquad \left| \phi''_{h, \tilde y}(\theta) \right| \leq \frac{|N|r}{d}\,.
    \end{equation}
    Integrating by parts in \eqref{eq ibp prep}, we obtain with \eqref{eq gamma der} and \eqref{eq phi der}
    \begin{align*}
        \eqref{eq ibp prep} &\leq \int\!\!\int\!\!\int \frac{|\Gamma_{h, \tilde y}'(\theta)|}{|\phi_{h, \tilde y}'(\theta)|} + \frac{|\Gamma_{h, \tilde y}(\theta)|}{|\phi_{h, \tilde y}'(\theta)|^2}|\phi_{h, \tilde y}''(\theta)|\, \mathrm{d}\theta h \,\mathrm{d}h \, \mathrm{d}\tilde y\\
        &\lesssim \int \! \! \int \!\! \int  (\mathbf{1}_{B(0,1)})_a(\tilde x - \tilde y, x_{d-1} - r\cos(\theta), x_d - r \sin(\theta), x_{d+1} - h) \\
        &\qquad\qquad\qquad 2^{-ds} (\frac{1}{ab|N|} + \frac{1}{b^2|N|})\mathbf{1}_{2^{s-3} < |y| < 2^s} \, \mathrm{d} \theta\, h \, \mathrm{d}h \, \mathrm{d}\tilde y\\  
        &\lesssim \int_{2^{s-3}< |y| < 2^{s}} 2^{-ds} \frac{1}{ab|N|}  (\mathbf{1}_{B(0,1)})_a(x' - y, x_{d+1} - |y|^2) \, \mathrm{d}y\\
        &\lesssim \frac{1}{ab|N|} \chi_{s,a}\,. 
    \end{align*}

    Now we assume that $V = \{0\}^d \times \R$. Then we have $N' = 0$ and hence
    \begin{multline}
    \label{eq ibp prep 2}
        \int \psi_{a}(x - z) e^{iN\cdot z} \, \mathrm{d}\mu_s(z)= \int \psi_{a}(x' - y, x_{d+1} - |y|^2) e^{iN_{d+1} |y|^2} K_s(y) \, \mathrm{d}y \\
        = \int \Gamma(t) e^{i N_{d+1}t} \, \mathrm{d}t\,,
    \end{multline}
    where this time we have set 
    $$
        \Gamma(t) = \int \psi_{a}(x' - y, x_{d+1} - t) K_s(y) \delta(t - |y|^2) \, \mathrm{d}y\,.
    $$
    Using \eqref{eq ks size} and that $s \geq -1$, one finds that $\left| \Gamma'(t)\right| \lesssim a^{-d-2} 2^{-ds}$. Integrating by parts in \eqref{eq ibp prep 2} gives the desired estimate \eqref{eq osc aux 1} after a similar, but simpler computation as in the case $V = \R^{d-1} \times \{0\}^2$.
\end{proof}

\begin{lemma}
    \label{lem sep tree aux 2}
    There exists $C > 0$ such that the following holds.
    Let $s \in \mathbb{Z}$, $a > 0$ and $b > 2^{-s}$. Let $\psi, \varphi \in \Psi$ with $\hat \psi(-aN) = 0$. Define 
    $$
        \varphi_{b,s}(x) = (2^sb)^{-d-2} \varphi(2^{-s}b^{-1} x', 2^{-2s}b^{-2} x_{d+1})
    $$ 
    and $\mu_s^l = \mu_s * \varphi_{b,s}$. Then we have that
    \begin{equation}
    \label{eq musl osc bound}
        |\psi_{a} * (e^{iN \cdot y} \mu_s^l(y))|\leq C
        a 2^{-s} b^{-(d+3)} 2^{-(d+2)s} \mathbf{1}_{[-2^s, 2^s]^d \times [-2^{2s}, 2^{2s}]}(y)\,.
    \end{equation}
\end{lemma}

\begin{proof}
    We have that  
    \begin{align*}
        |\psi_{a} * e^{iN \cdot y} \mu_s^l(y)|
        &\lesssim \int |\hat \psi(a \xi - a N)| |\hat \varphi_{b,s}(\xi)| \, \mathrm{d}\xi\\
        &\lesssim a \int |\xi||\hat \varphi_{b,s}(\xi)| \, \mathrm{d}\xi\\
        &= a 2^{-(d+3)s} b^{-(d+3)} \int |(\xi', 2^{-s} b^{-1} \xi_{d+1})| |\hat \varphi(\xi)| \, \mathrm{d}\xi\,.
    \end{align*}
    The lemma now follows since $2^{-s}b^{-1} < 1$ and $\||\xi|\hat \varphi(\xi)\|_1 \lesssim 1$ for $\psi \in \Psi$, and since the left hand side of \eqref{eq musl osc bound} is clearly supported in $[-2^s, 2^s]^d \times [-2^{2s}, 2^{2s}]$. 
\end{proof}

\subsection{Main almost orthogonality estimate for separated trees}
\label{sub sec sep tree}

\begin{lemma}
    \label{lem sep trees}
    There exists $C > 0$ such that the following holds. 
    Let $\mathbf{T}_1, \mathbf{T}_2$ be a pair of $\Delta$-separated, normal trees. Then 
    $$
        \left\lvert \int T_{\mathbf{T}_1}^* g_1 \overline{T_{\mathbf{T}_2}^* g_2}\right\rvert \leq C \Delta^{-\frac{1}{10d}} \| W_{\mathbf{T}_1} g_1 \|_{L^2(I(\mathbf{T}_1) \cap I(\mathbf{T}_2))}  \| W_{\mathbf{T}_2} g_2 \|_{L^2(I(\mathbf{T}_1) \cap I(\mathbf{T}_2))} 
    $$
    for certain operators $W_{\mathbf{T}_i}$ depending only on $\mathbf{T}_i$ with 
    $$
        \|W_{\mathbf{T}_i}\|_{2 \to 2} \leq 1\,.
    $$
\end{lemma}

\begin{proof}
    Since the trees are normal, the left hand side vanishes if $I(\mathbf{T}_1) \cap I(\mathbf{T}_2) = \emptyset$. Thus we may assume without loss of generality that $I := I(\mathbf{T}_2) \subset I(\mathbf{T}_1)$.
    Note also that we can always assume that $\Delta$ is sufficiently large, by adding $|T_{\mathbf{T}_i}^*|$ to $W_{\mathbf{T}_i}$. Finally, we assume that the minimal scale of a tile in $\mathbf{T}_2$ is $0$, by scaling.

    We fix $\gamma := \frac{2(d+1)}{(2d+1)(2d+7)}> \frac{2}{10d} > 0$ and $\varphi \in \Psi$ and define $\mu_s^{l} := \mu_s * \varphi_{\Delta^{-\gamma},s}$, where 
    $$
        \varphi_{\Delta^{-\gamma},s}(x) = 2^{-(d+2)s}\Delta^{\gamma(d+2)} \varphi( 2^{-s} \Delta^\gamma x', 2^{-2s} \Delta^{2\gamma} x_{d+1})\,.
    $$
    Then we decompose $T_{\mathbf{T}_1}$ into a large scales part, a smoothened small scales part and an error term:
    \begin{align*}
        T_{\mathbf{T}_1} f &= \sum_{\substack{s \in \sigma(x)\\ s> \overline{\sigma}(x) -\log_2(\Delta)}} \int f(x-y) e^{iN(x)\cdot y} \, \mathrm{d}\mu_s(y)\\
        &\quad+ \sum_{\substack{s \in \sigma(x)\\ s \leq \overline{\sigma}(x) - \log_2(\Delta)}} \int f(x-y) e^{iN(\mathbf{T}_1)\cdot y} \, \mathrm{d}\mu_s^{l}(y)\\
        &\quad+ \mathcal{E}_{\mathbf{T}_1}(f)\\
        &=: T_{\mathbf{T}_1}^{1} f + T_{\mathbf{T}_1}^{2} f + \mathcal{E}_{\mathbf{T}_1}(f)\,.
    \end{align*}
    Notice that $\mu_s, \lambda_s$ are supported in $\delta_s(B(0,1/3))$, so for $\Delta$ sufficiently large the measures $\mu_s^l$ and $\lambda_s^l$ also satisfy conditions i) to v). Thus both $T_{\mathbf{T}_1}^1$ and $T_{\mathbf{T}_1}^2$ are generalized tree operators.
    
    We first show that the error term $\mathcal{E}_{\mathbf{T}_1}$ is small. We have
    \begin{align}
        |\mathcal{E}_{\mathbf{T}_1}(f)| &\leq 
        \left\lvert\sum_{\substack{s \in \sigma(x)\\ s \leq \overline{\sigma}(x) - \log_2(\Delta)}} \int f(x-y) (e^{iN(x)\cdot y} - e^{iN(\mathbf{T}_1)\cdot y}) \, \mathrm{d}\mu_s(y)\right\rvert\nonumber\\
        &\quad+ \left\lvert\sum_{\substack{s \in \sigma(x)\\ s \leq \overline{\sigma}(x) - \log_2(\Delta)}} \int f(x-y) e^{iN(\mathbf{T}_1)\cdot y} \, \mathrm{d}(\mu_s - \mu_s^l)(y)\right\rvert\nonumber\\
        &\lesssim \sum_{\substack{s \in \sigma(x)\\ s \leq \overline{\sigma}(x) - \log_2(\Delta)}} 2^{s - \overline{\sigma}(x)}\int |f(x-y)|  \, \mathrm{d}|\mu_s|(y)\nonumber\\
        &\quad+ \sup_{\underline{s} < \overline{s}} \left\lvert \sum_{\substack{\underline{s} < s \leq \overline{s}\\ s \in J}} (\mu_s - \mu_s^l) * (f(y) e^{iN(\mathbf{T}_1)\cdot y}) \right\rvert\,.\nonumber
    \end{align}
    Here $J = J(\mathbf{T}_1)$ is as in Subsection \ref{sub sec intro tree}.
    The first term is bounded by $\Delta^{-1} M^{\lambda} f(x)$. The second term is the maximally truncated singular integral associated to the single scale operators $\mathbf{1}_J(s)(\mu_s - \mu_s^l)$, applied to $e^{iN(\mathbf{T}_1)\cdot y}f(y)$. Since 
    $$
        |\hat \mu_s(\xi) - \hat \mu_s^l(\xi)| \leq \min\{|\xi|^{-1}, |\Delta^{-\gamma}\xi|\} \,,
    $$
    this operator has norm $\lesssim \Delta^{-\gamma/2}$ on $L^2$. Thus we have 
    $$
        \|\mathcal{E}_{\mathbf{T}_1}(f)\|_2 \lesssim \Delta^{-\gamma/2} \|f\|_2 \leq \Delta^{-\frac{1}{10d}}\|f\|_2\,.
    $$
    
    It now remains to bound
    $\langle T_{\mathbf{T}_1}^{1*} g_1 ,T_{\mathbf{T}_2}^* g_2\rangle + \langle T_{\mathbf{T}_1}^{2*} g_1 ,T_{\mathbf{T}_2}^* g_2\rangle$.
    We decompose $\mathbf{T}_1 = \mathbf{T}_1' \cup \mathbf{T}_1'' \cup \mathbf{T}_1'''$,
    where 
    $$
        \mathbf{T}_1' = \{p \in \mathbf{T}_1 \, : \, s(p) < - 1\,, 3I(p) \subset I\}\,,
    $$
    $$
        \mathbf{T}_1'' = \{p \in \mathbf{T}_1 \, : \, s(p) \geq  - 1,\,3I(p) \cap I \neq \emptyset\}
    $$
    and $\mathbf{T}_1''' = \mathbf{T}_1 \setminus (\mathbf{T}_1' \cup \mathbf{T}_1'')$. 
    Then $\mathbf{T}_1'$ is a normal tree with top $I$. For all $p \in \mathbf{T}_1''', p' \in \mathbf{T}_2$ we have $2I(p) \cap (2 + \frac{1}{4})I(p') = \emptyset$, this follows from normality of $\mathbf{T}_2$ and the fact that all tiles in $\mathbf{T}_1'''$ with $3I(p) \cap I \neq \emptyset$ have scale at most $s(p') - 2$. Thus $\langle T_{\mathbf{T}_1'''}^* g_1, T_{\mathbf{T}_2}^*g_2\rangle = 0$. 
    
    According to the decomposition of $\mathbf{T}_1$ we split $T_{\mathbf{T}_1}^1 = T_{\mathbf{T}_1'}^1+ T_{\mathbf{T}_1''}^1 + T_{\mathbf{T}_1'''}^1$, and similarly for $T_{\mathbf{T}_1}^2$. We still have $\langle T_{\mathbf{T}_1'''}^{1*} g_1, T_{\mathbf{T}_2}^*g_2\rangle = 0$, thus
    \begin{equation}
        \label{eq T1 dec}
        \langle T_{\mathbf{T}_1}^{1*} g_1, T_{\mathbf{T}_2}^* g_2\rangle = \langle T_{\mathbf{T}_1'}^{1*} g_1, T_{\mathbf{T}_2}^* g_2\rangle + \langle T_{\mathbf{T}_1''}^{1*} g_1, T_{\mathbf{T}_2}^* g_2\rangle\,.
    \end{equation}
    For the first summand in \eqref{eq T1 dec} we obtain with Corollary \ref{cor sep trees baby}
    \begin{equation}
        \label{eq bound t1 1}
        |\langle T_{\mathbf{T}_1'}^{1*} g_1, T_{\mathbf{T}_2}^* g_2\rangle| \lesssim \Delta^{-\frac{1}{10d}} \|g_1\|_{L^2(I)}\|g_2\|_{L^2(I)}\,.
    \end{equation}
    For the second summand in \eqref{eq T1 dec} we put $\nu = \frac{d+1}{2d+1}$ and choose $\Pi_2 = \Pi_{\Delta^\nu, N(\mathbf{T}_2)}$ for some function $\psi \in \Psi$. If $\Delta$ is sufficiently large then $\Pi_2 T_{p'}^* g_2$ is supported in $(2 + \frac{1}{4})I(p')$ for each tile $p' \in \mathbf{T}_2$, thus we also have $\langle T_{\mathbf{T}_1'''}^{1*} g_1, \Pi_2 T_{\mathbf{T}_2}^*g_2\rangle = 0$. Using this, we expand
    \begin{equation}
        \label{eq T1' dec}
        \langle T_{\mathbf{T}_1''}^{1*} g_1, T_{\mathbf{T}_2}^* g_2\rangle = \langle g_1, T_{\mathbf{T}_1''}^{1}  (1 - \Pi_2)T_{\mathbf{T}_2}^{*} g_2\rangle + \langle T_{\mathbf{T}_1}^{1*}g_1 - T_{\mathbf{T}_1'}^{1*} g_1, \Pi_2 T_{\mathbf{T}_2}^* g_2\rangle\,.
    \end{equation}
    The second term in \eqref{eq T1' dec} is by Cauchy-Schwarz, Lemma \ref{lem T* proj} and $L^2$-boundedness of $T_{\mathbf{T}_1'}^{1*}=T_{\mathbf{T}_1'}^{1*} \mathbf{1}_{I(\mathbf{T}_2)}$ bounded by 
    \begin{equation}
        \label{eq bound t1 2}
         \Delta^{-\frac{1}{10d}} \||T_{\mathbf{T}_1}^{1*}g_1|+|g_1|\|_{L^2(I)} \|g_2\|_{L^2(I)}\,.
    \end{equation}
    For the first term in \eqref{eq T1' dec} we may assume that $N(\mathbf{T}_2) = 0$. Then we have
    $$
        T_{\mathbf{T}_1''}^1(1 - \Pi_2) f (x) = \sum_{\substack{s \in \sigma'(x)\\ s> \overline{\sigma}(x) -\log_2(\Delta)}}  (e^{iN(x) \cdot y} \mu_s(y))*\psi_{\Delta^{-\nu}} *f (x)\,,
    $$
    where $\sigma'(x) = \{s(p) \, : \,x \in E(p), p \in \mathbf{T}_1''\}$.
    Using the estimate for the convolution kernels proven in Lemma \ref{lem sep tree aux} with $a = \Delta^{-\nu}$, $b = \Delta^{-1/2}$, we obtain with the notation defined at \eqref{eq chi def}
    \begin{multline*}
        |(e^{iN(x) \cdot y} \mu_s(y))*\psi_{\Delta^{-\nu}} *f (x)| \\
        \lesssim (\Delta^{-1/2} \chi_{s, \Delta^{-\nu}, \Delta^{-1/2}} + \Delta^{\nu - 1/2} \chi_{s, \Delta^{-\nu}})*|f|(x)\,.
    \end{multline*}
    Passing the convolutions to the other side in the inner product, we find that the first term in \eqref{eq T1' dec} is bounded by 
    \begin{multline*}
        \log(\Delta) \Delta^{\nu-1/2} \langle M^{1,\chi}g_1 + M^{2,\chi}g_1, |T_{\mathbf{T}_2}^* g_2|\rangle\\\lesssim \Delta^{-\frac{1}{10d}} \| M^{1,\chi}g_1 + M^{2,\chi}g_1\|_{L^2(I)} \|g_2\|_{L^2(I)}
    \end{multline*}
    where $M^{1,\chi}$, $M^{2,\chi}$ are defined in \eqref{eq M1 def} and \eqref{eq M2 def}. Since $M^{1,\chi}$ and $M^{2,\chi}$ are bounded on $L^2$, this gives the desired estimate for $T_{\mathbf{T}_1}^1$. 

    Now we turn to $\langle T_{\mathbf{T}_1}^{2*} g_1, T_{\mathbf{T}_2}^* g_2\rangle $. As above, we have
    \begin{equation}
        \label{eq T2 dec}
        \langle T_{\mathbf{T}_1}^{2*} g_1, T_{\mathbf{T}_2}^* g_2\rangle = \langle T_{\mathbf{T}_1'}^{2*} g_1, T_{\mathbf{T}_2}^* g_2\rangle + \langle T_{\mathbf{T}_1''}^{2*} g_1, T_{\mathbf{T}_2}^* g_2\rangle\,.
    \end{equation}
    For the first term in \eqref{eq T2 dec} we have again from Corollary \ref{cor sep trees baby} that 
    \begin{equation}
        \label{eq bound t2 1}
        |\langle T_{\mathbf{T}_1'}^{2*} g_1, T_{\mathbf{T}_2}^* g_2\rangle| \lesssim \Delta^{-\frac{1}{10d}} \|g_1\|_{L^2(I)}\|g_2\|_{L^2(I)}\,.
    \end{equation}
    For the second term in \eqref{eq T2 dec} we split as before
    \begin{equation}
    \label{eq T2' dec}
        \langle T_{\mathbf{T}_1''}^{2*} g_1, T_{\mathbf{T}_2}^* g_2\rangle = \langle g_1, T_{\mathbf{T}_1''}^{2}  (1 - \Pi_2)T_{\mathbf{T}_2}^{*} g_2\rangle + \langle T_{\mathbf{T}_1}^{2*}g_1 - T_{\mathbf{T}_1'}^{2*} g_1, \Pi_2 T_{\mathbf{T}_2}^* g_2\rangle\,.
    \end{equation}
    The second term in \eqref{eq T2' dec} is by Cauchy-Schwarz and Lemma \ref{lem T* proj} bounded by 
    \begin{equation}
        \label{eq bound t2 2}
         \Delta^{-\frac{1}{10d}} \||T_{\mathbf{T}_1}^{2*}g_1|+|g_1|\|_{L^2(I)} \|g_2\|_{L^2(I)}\,.
    \end{equation}
    For the first term in \eqref{eq T2' dec}, we  assume again that $N(\mathbf{T}_2) = 0$. 
    Then 
    \begin{align*}
        T_{\mathbf{T}_1''}^2(1 - \Pi_2) f (x) = \sum_{\substack{s \in \sigma'(x)\\ s \leq \overline{\sigma}(x) - \log_2(\Delta)}} \psi_{\Delta^{-\nu}} * (e^{iN(\mathbf{T}_1) \cdot y} \mu_s^l(y)) * f(x)\,.
    \end{align*}
    By our assumptions on $\Psi$, we may then further assume that we have $\hat \psi(\Delta^{-\nu}N(\mathbf{T}_1)) = 0$, since for $\Delta$ sufficiently large $|\Delta^{-\nu} N(\mathbf{T}_1)| \gtrsim \Delta^{1-\nu} \geq 1$. 
    By Lemma \ref{lem sep tree aux 2} with $a = \Delta^{-\nu}$, $b = \Delta^{-\gamma}$ we have
    \begin{multline*}
        \sum_{s > \gamma \log_2(\Delta)} |\psi_{\Delta^{-\nu}} * (e^{iN(\mathbf{T}_1)\cdot y}\mu_s^l(y)) * f(x)| \lesssim \Delta^{(d+3)\gamma - \nu} \Phi * |f|(x)\,,
    \end{multline*}
    where $\Phi(y) = \sum_{s \geq 0} 2^{-s(d+3)} \mathbf{1}_{[-2^s, 2^s]^d \times [-2^{2s}, 2^{2s}]}(y)$.
    By Lemma \ref{lem sep tree aux} with $a = \Delta^{-\nu}$, $b = \Delta^{-1/2}$ we have for $-1 \leq s \leq \gamma \log(\Delta)$ that 
    \begin{multline*}
        |\psi_{\Delta^{-\nu}} * (e^{iN(\mathbf{T}_1 \cdot y}\mu_s^l(y)) * f(x)|\\ \lesssim (\Delta^{-1/2} \chi_{s, \Delta^{-\gamma}, \Delta^{-1/2}} + \Delta^{\nu - 1/2} \chi_{s, \Delta^{-\gamma}})*M|f|(x)\,.
    \end{multline*}
    Similarly as for $T_{\mathbf{T}_1}^1$, we obtain from these estimates that the first term in \eqref{eq T2' dec} is bounded by 
    $$
        \Delta^{-\frac{1}{10d}} \|Mg_1 + M^{1,\chi}g_1 + M^{2,\chi} g_1\|_{L^2(I)}\|g_2\|_{L^2(I)}\,.
    $$
    This completes the proof.
\end{proof}

\subsection{Completing the argument for forests}
\label{sub sec comp for}

Here we use Lemma \ref{lem sep trees} to complete the proof of Proposition \ref{prop forest}.
We follow the presentation in \cite{ZK2021}.

A row is a union of normal trees with pairwise disjoint top cubes.

\begin{lemma}
    \label{lem rows}
    There exists $C > 0$ such that the following holds.
    Let $\mathbf{R}_1, \mathbf{R}_2$ be two rows such that the trees in $\mathbf{R}_1$ are $\Delta$-separated from the trees in $\mathbf{R}_2$. Then 
    $$
        \left\lvert \int T_{\mathbf{R}_1}^* g_1 \overline{T_{\mathbf{R}_2}^* g_2} \right\rvert \leq C \Delta^{-\frac{1}{10d}} \|g_1\|_2 \|g_2\|_2\,.
    $$
\end{lemma}

\begin{proof}
    We have by Lemma \ref{lem sep trees}, with $W_{\mathbf{T}_i}$ as defined there:
    \begin{multline*}
         \left\lvert \int T_{\mathbf{R}_1}^* g_1 \overline{T_{\mathbf{R}_2}^* g_2} \right\rvert 
         \leq \sum_{\mathbf{T}_1 \in \mathbf{R}_1} \sum_{\mathbf{T}_2 \in \mathbf{R}_2}  \left\lvert \int T_{\mathbf{T}_1}^* \mathbf{1}_{I(\mathbf{T}_1)}g_1 \overline{T_{\mathbf{T}_2}^* \mathbf{1}_{I(\mathbf{T}_2)}g_2} \right\rvert\\
        \leq C \Delta^{-\frac{1}{10d}} \sum_{\mathbf{T}_1 \in \mathbf{R}_1} \sum_{\mathbf{T}_2 \in \mathbf{R}_2} \prod_{i = 1,2} \|W_{\mathbf{T}_i} \mathbf{1}_{I(\mathbf{T}_i)} g_i\|_{L^2(I(\mathbf{T}_1)\cap I(\mathbf{T}_2))}\,.
    \end{multline*}
    Using Cauchy-Schwarz, disjointness of the cubes $I(\mathbf{T}_i)$ for $\mathbf{T}_i \in \mathbf{R}_i$ and $\|W_{\mathbf{T}_i}\|_{2 \to 2} \lesssim 1$, we estimate:
    \begin{align*}
        &\leq C\Delta^{-\frac{1}{10d}} \prod_{i=1,2}\left(\sum_{\mathbf{T}_1 \in \mathbf{R}_2} \sum_{\mathbf{T}_2 \in \mathbf{R}_2} \|W_\mathbf{T_i} \mathbf{1}_{I(\mathbf{T}_i)} g_i\|^2_{L^2(I(\mathbf{T}_1) \cap I(\mathbf{T}_2))} \right)^{1/2}\\
        &\leq C\Delta^{-\frac{1}{10d}} \prod_{i=1,2}\left(\sum_{\mathbf{T}_i \in \mathbf{R}_i}  \|W_\mathbf{T_i} \mathbf{1}_{I(\mathbf{T}_i)} g_i\|^2_{L^2} \right)^{1/2}\\
        &\leq C \Delta^{-\frac{1}{10d}} \|g_1\|_2 \|g_2\|_2\,.\qedhere
    \end{align*}
\end{proof}

\begin{proof}[Proof of Proposition \ref{prop forest}]   
    Let $\mathbf{F}$ be an $n$-forest.
    By the overlap estimate \eqref{eq forest overlap}, we can decompose  $\mathbf{F}$ into $2^n \log(n+2)$ rows $\mathbf{R}_i$. Since $\mathbf{F}$ is an $n$-forest, the trees in different rows are $2^{10dn}$ separated and have density at most $2^{-n}$. Note also that separation of the rows implies that the sets $E(\mathbf{R}_i) := \cup_{p \in \mathbf{R}_i} E(p)$ are pairwise disjoint. 
    
    By orthogonality and Corollary \ref{cor single tree}, we have that
    \begin{equation*}
        \label{eq row est}
        \|T_{\mathbf{R}_i}\|_{2\to2} \lesssim_\varepsilon 2^{-\epsilon n}
    \end{equation*}
    for each $i$ and each $\varepsilon < \frac{1}{2} - \frac{1}{2(d+1)}$.  Using this and Lemma \ref{lem rows} yields
    \begin{multline}
    \label{eq final forest}
        \|T_\mathbf{F}^* g\|_2^2
        \leq \sum_{i,j} \left\lvert \int T_{\mathbf{R}_i}^* \mathbf{1}_{E(\mathbf{R}_i)} g T_{\mathbf{R}_j}^* \mathbf{1}_{E(\mathbf{R}_j)}g \right\rvert\\
        \lesssim (2^{-2\varepsilon n}+ 2^n \log(n+2) 2^{-\frac{1}{10d} 10 d n}) \sum_i \|\mathbf{1}_{E(\mathbf{R}_i)}g\|_2^2\\ \lesssim 2^{-2\varepsilon n} \|g\|_2^2\,.
    \end{multline}
    This completes the proof.
\end{proof}

\section{\texorpdfstring{$L^p$}{Lp}-bounds: Proof of Theorem \ref{main Lp}}
\label{sec Lp}

Theorem  \ref{main Lp} follows from interpolation and the estimate
\begin{equation}
    \label{eq RWT}
    |\langle \mathbf{1}_F, T_{\mathbf{P}_{\mathrm{ad}}} \mathbf{1}_G\rangle| \le C_p |F|^{1/p'} |G|^{1/p}
\end{equation}
for all $F, G \subset \R^{d+1}$ with $0 < |F|, |G| < \infty$ and all $p$ satisfying
\begin{equation}
    \label{eq p cond prop6}
    \frac{d^2 + 4d + 2}{(d+1)^2} < p < 2(d+1)\,.
\end{equation}

Before verifying \eqref{eq RWT} we record an $L^p$ estimate for antichains.

\begin{proposition}
    \label{prop antichain Lp}
    For each $1 < p < \infty$ there exists $\varepsilon = \varepsilon(d,p) > 0$ and $C > 0$ such that for every antichain $\mathbf{A}$ of density $\delta$
    $$
        \|T_{\mathbf{A}}\|_{p \to p} \leq C \delta^{-\varepsilon}\,. 
    $$
\end{proposition}

\begin{proof}
    This follows from interpolation between Proposition \ref{prop antichain} and the trivial bound $\|T_\mathbf{A}\|_{p \to p} \leq C(p)$, which holds since $T_\mathbf{A}$ is dominated pointwise by the maximal average along the paraboloid.
\end{proof}

\begin{proof}[Proof of \eqref{eq RWT}]
    Suppose first that $p < 2$. Fix sets $F$ and $G$ with $0 < |F|, |G| < \infty$.
    We run the constructions in Section \ref{sec outline} and \ref{sec organization} as in the proof of weak $L^2$ bounds and define
    \[
        \bar F = \tilde F \cup \{M \mathbf{1}_G > C|G|/|F|\}
    \]
    where $\tilde F$ is the set from Proposition \ref{prop organization}. Choosing $C$ sufficiently large, we have $|\bar F| \le 3|F|/4$.

    We claim that for each forest $\mathbf{F}_{n,l}$ and for all $0 \leq \alpha < d^2/(2(d^2 + 4d + 2))$ there exists $\varepsilon > 0$ such that
    \begin{equation}
        \label{eq localized}
        \|\mathbf{1}_{F \setminus \bar F} T_{\mathbf{F}_{n,l}} \mathbf{1}_G\|_{2 \to 2} \lesssim (|G||F|^{-1})^\alpha 2^{-\varepsilon n}\,.
    \end{equation}
    Suppose the claim \eqref{eq localized} holds. Then 
    \begin{equation}
        \label{eq single forest loc}
        |\langle \mathbf{1}_{F \setminus \bar F}, T_{\mathbf{F_{n,l}}} \mathbf{1}_G \rangle| \lesssim 2^{-\epsilon n} |F|^{1/2 - \alpha} |G|^{1/2 + \alpha}.
    \end{equation}
    Using the decomposition of $\mathbf{P}_\mathrm{ad}$ into forests and antichains from Proposition \ref{prop organization}, using Proposition \ref{prop antichain Lp} for antichains and \eqref{eq single forest loc} for forests and summing in $n$ yields
    \[
        |\langle \mathbf{1}_{F \setminus \bar F}, T_{\mathbf{P_\mathrm{ad}}} \mathbf{1}_G \rangle| \lesssim |F|^{1/2 - \alpha} |G|^{1/2 + |\alpha|}.
    \]
    Finally, defining $F_0 = F$ and iteratively $F_{i+1} = \bar F_i$, we have $|F_i| \le (3/4)^i|F|$ and hence
    \begin{align}
        |\langle \mathbf{1}_{F}, T_{\mathbf{P_\mathrm{ad}}} \mathbf{1}_G \rangle| &= \sum_{i \ge 0} |\langle \mathbf{1}_{F_i \setminus F_{i+1}}, T_{\mathbf{P_\mathrm{ad}}} \mathbf{1}_G \rangle|\nonumber\\
        &\lesssim \sum_{i \ge 0} (3/4)^{(1/2 - \alpha)i} |F|^{1/2 - \alpha}|G|^{1/2 + \alpha}\nonumber\\
        &\lesssim |F|^{1/2 - \alpha} |G|^{1/2 + \alpha}. \label{eq iteration}
    \end{align}
    This completes the proof of \eqref{eq RWT}. This extrapolation argument is attributed to Bateman \cite{Bateman_2013} in \cite{ZK2021}.

    Now we verify the claim \eqref{eq localized}. Recall that by Lemma \ref{lem sparse} for each tree $\mathbf{T}$, all $p > 1+\frac{1}{d+1}$ and all $f,g$, there exists a sparse collection $\mathcal{S}$ with
    \begin{equation}
    \label{eq loc sparse}
        |\langle \mathbf{1}_{F \setminus \bar F} T_{*\mathbf{T}} \mathbf{1}_G f, g \rangle| \lesssim \sum_{Q \in \mathcal{S}} |Q| \langle \mathbf{1}_G f \rangle_{p,Q} \langle \mathbf{1}_{(F \setminus \bar F) \cap E(Q)} g \rangle_{p, 3Q}\,.
    \end{equation}
    If $3Q \cap (F \setminus \bar F) \neq \emptyset$, then $|Q \cap G| \lesssim |G||F|^{-1} |Q|$. Using this and Hölder's inequality, we obtain with a similar argument as in the proof of Corollary \ref{cor sparse} that for each $\varepsilon < \frac{1}{2} - \frac{1}{d+2}$ and each tree $\mathbf{T}$ of density at most $2^{-n}$
    $$
        \|\mathbf{1}_{F \setminus \bar F} T_{*\mathbf{T}}\mathbf{1}_G\|_{2 \to 2} \lesssim (|G||F|^{-1})^{\varepsilon} 2^{-\varepsilon n}\,.
    $$
    By the same argument, the same holds for $M_{\mathbf{T}}$ and hence for $T_{\mathbf{T}}$.
    By orthogonality, the same estimate holds for row operators $T_\mathbf{R}$. Finally, each forest $\mathbf{F}_{n,l}$ can be decomposed into at most $2^n \log(n+2)$ rows, so 
    $$
        \|\mathbf{1}_{F \setminus \bar F} T_{\mathbf{F}_{n,l}}\mathbf{1}_G\|_{2 \to 2} \lesssim 2^{n/2 - \varepsilon n} \log(n+2)^{1/2} (|G||F|^{-1})^{\varepsilon}\,.
    $$
    We obtain \eqref{eq localized} by taking a geometric average of this estimate and \eqref{eq final forest}.

    We turn to the case $p > 2$. We fix again $F, G$ and run the constructions from Section  \ref{sec outline} and \ref{sec organization} as in the proof of weak $L^2$ bounds. Let $\tilde F$ be the exceptional set from Proposition \ref{prop organization}, it depends only on $F$. Let
    \[
        \tilde G = G \cap \{M\mathbf{1}_{F} > C|F|/|G|\},
    \]
    where $C$ is chosen sufficiently large so that $|\tilde G| \le |G|/2$.
    We claim that for all $0 \leq \alpha < \frac{1}{2} - \frac{1}{2(d+1)}$
    \begin{equation}
        \label{eq localized 2}
        \| \mathbf{1}_{F \setminus \tilde F} \sum_{n \geq 0} \sum_{l = 1}^{(n+1)^2} T_{\mathbf{F}_{n,l}} \mathbf{1}_{G \setminus \tilde G}\|_{2 \to 2} \leq C_\alpha (|F||G|^{-1})^\alpha\,.
    \end{equation}
    Suppose that \eqref{eq localized 2} holds. As in the case $p < 2$, we conclude using the decomposition of $\mathbf{P}_{\mathrm{ad}}$ into antichains and forests from Proposition \ref{prop organization} and using Proposition \ref{prop antichain Lp} that 
    \[
        |\langle \mathbf{1}_{F \setminus \tilde F}, T_{\mathbf{P}_{\mathrm{ad}}} \mathbf{1}_{G \setminus \tilde G}\rangle| \lesssim |F|^{1/2 + \alpha} |G|^{1/2 - \alpha}.
    \]
    Iterating over $G$ and summing a geometric series as in \eqref{eq iteration}, and then iterating over $F$ in the same way gives the restricted weak type bounds \eqref{eq RWT}.

    It remains to verify the claim \eqref{eq localized 2}. If a tile $p$ contributes in \eqref{eq localized 2}, then $3I(p) \cap (G \setminus \tilde G) \ne \emptyset$. Using the definition of $\tilde G$ it follows that for every cube $Q \supset I(p)$
    \[
        |Q \cap F| \lesssim |Q| (|F||G|^{-1}).
    \]
    If $p$ gets assigned to a forest $\mathbf{F}_{n,l}$ in the tile organization in Section \ref{sec organization}, then $p \in \mathbf{P}_k$ for some $k \le n$. By the definition of $\mathbf{P}_k$ there then exists a dyadic cube $Q \in \mathbf{D}_k(F)$ with $I(p) \subset Q$. This cube satisfies
    \[
        2^{-k-1} \le |Q \cap F|/|Q| \lesssim |F| |G|^{-1}.
    \]
    Hence, we have
    \begin{equation*}
        \sum_{n \geq 0} \sum_{l = 1}^{(n+1)^2} T_{\mathbf{F}_{n,l}} \mathbf{1}_{G \setminus \tilde G}= \sum_{n \ge C + \log(|G||F|^{-1})} \sum_{l = 1}^{(n+1)^2} T_{\mathbf{F}_{n,l}} \mathbf{1}_{G \setminus \tilde G}.
    \end{equation*}
    Now the claim \eqref{eq localized 2} follows from Proposition \ref{prop forest}.
\end{proof}

\printbibliography
\end{document}